\documentclass[10pt]{article}
   \usepackage{graphicx}
   \usepackage{epsfig}
   \usepackage{amssymb}
   \usepackage{amsmath}
   \usepackage{amsthm}
   \newtheorem{definition}{Definition}[section]
   \newtheorem{lemma}{Lemma}[section]
   \newtheorem{theorem}{Theorem}[section]

   \newtheorem{remark}{Remark}[section]

   {\end{description}}

   {\hfill $\bullet$ \\}

   \newcommand{\be}{\begin{equation}}
   \newcommand{\ee}{\end{equation}}

\begin{document}
    \title{A robust time-split linearized explicit/implicit technique for two-dimensional hydrodynamic model: an application to floods in Cameroon far north region}
   \author{Eric Ngondiep$^{\text{\,a\,b}}$ \thanks{\textbf{Email address:} ericngondiep@gmail.com}}
   \date{$^{\text{\,a\,}}$\small{Department of Mathematics and Statistics, College of Science, Imam Mohammad Ibn Saud\\ Islamic University
        (IMSIU), $90950$ Riyadh $11632,$ Saudi Arabia.}\\
     \text{\,}\\
       $^{\text{\,b\,}}$\small{Hydrological Research Center, Institute for Geological and Mining Research, 4110 Yaounde-Cameroon.}}
   \maketitle

   \textbf{Abstract.}
    This paper deals with a time-split explicit/implicit approach for solving a two-dimensional hydrodynamic flow model with appropriate initial and boundary conditions. The time-split technique is employed to upwind the convection term and to treat the friction slope so that the numerical oscillations and stability are well controlled. A suitable time step restriction for stability and convergence accurate of the new algorithm is established using the $L^{\infty}(0,T; L^{2})$-norm. Under a time step requirement, some numerical examples confirm the theoretical studies and suggest that the proposed computational technique is spatial fourth-order accurate and temporal second-order convergent. An application to floods observed in Cameroon far north region is considered and discussed.\\
    \text{\,}\\

   \\
   {\bf AMS Subject Classification (MSC): 65M12, 65M06}.

  \section{Introduction and motivation}\label{sec1}

   \text{\,\,\,\,\,\,\,\,\,\,}Efficient computational approaches to forecast floods can provide useful information for water resource management and inundation risk mitigation. Flood estimations such as water depth and spatial flood extent are crucial tools (information) allowing the stake holders to develop important strategies and make useful decisions in future flood management and risk prevention strategy design. Additionally, inundations are very important in natural water cycle which enables the experts to construct efficient computational tools to assess healthy ecosystems \cite{hernonin2013,yu2014}. Researchers are sometimes confronted with the challenge of forecasting the timing and magnitude of rainfall generated run-off from watersheds for pollution prevention, flood control and ecosystem goals. A component of the overland flow is the shallow water flow which results when the rainfall rate exceeds the soil infiltration capacity in some areas of the watershed (see Figure $\ref{fig1}$). However, the overland flow velocities and depths are highly discontinuous in time and space variables. This is due to a small natural variation in spaces of soil hydraulic properties and small-scale ground surface micro-topography. Because of the complexity of the hydrodynamic flow problem and the numerical difficulties in simulating this model, researchers have been obliged to simulate complex hillslopes as plane surfaces with constant hydraulic properties and Thacker's axisymmetrical solution (see Figure $\ref{fig1}$, \textbf{Figure 1.iv}), both solutions consider the kinematic wave approximation to full hydrodynamic equations. However, the small-scale dynamic interactions between the surface and ground surface flows are ignored since this technique does not deal explicitly with the space variable soil properties and micro-topography.\\

     \begin{figure}
         \begin{center}
         \begin{tabular}{c c}
         \psfig{file=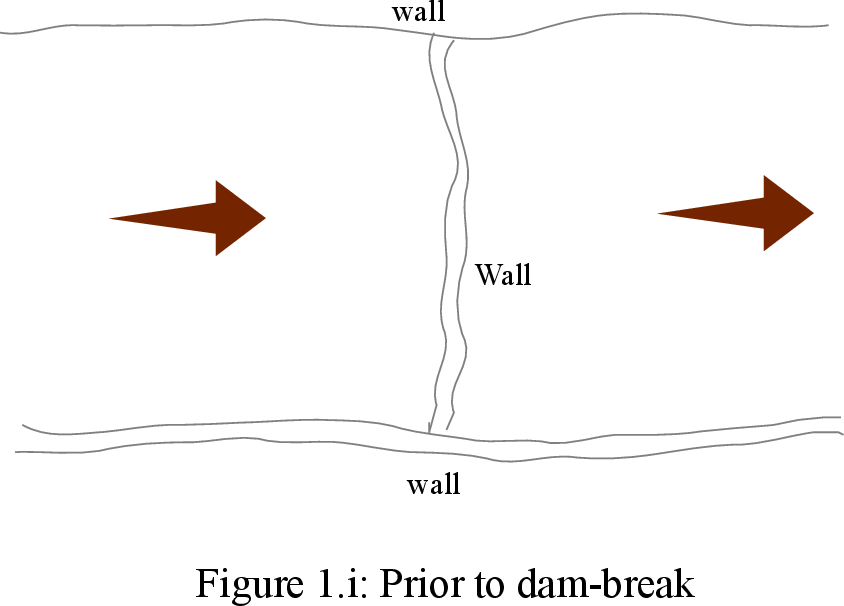,width=7cm} & \psfig{file=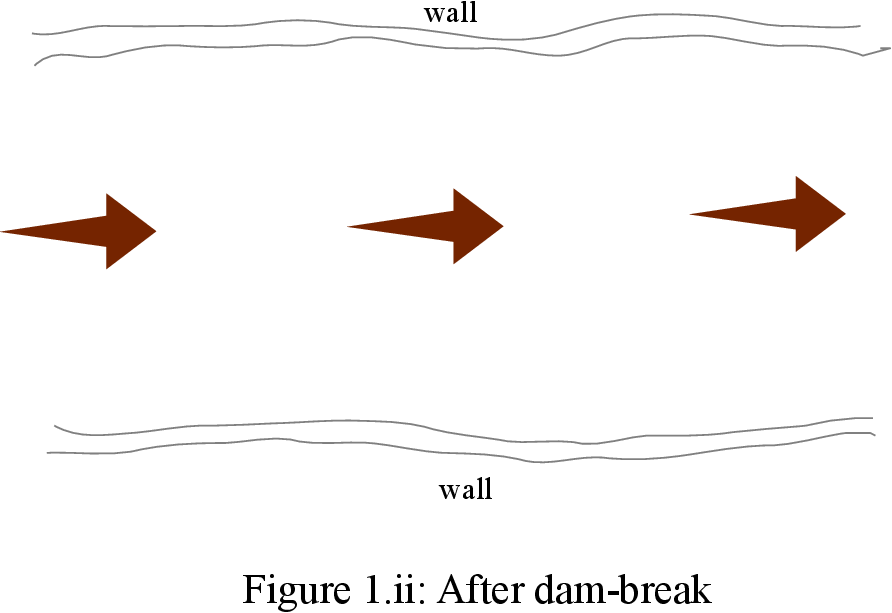,width=7cm}\\
         \psfig{file=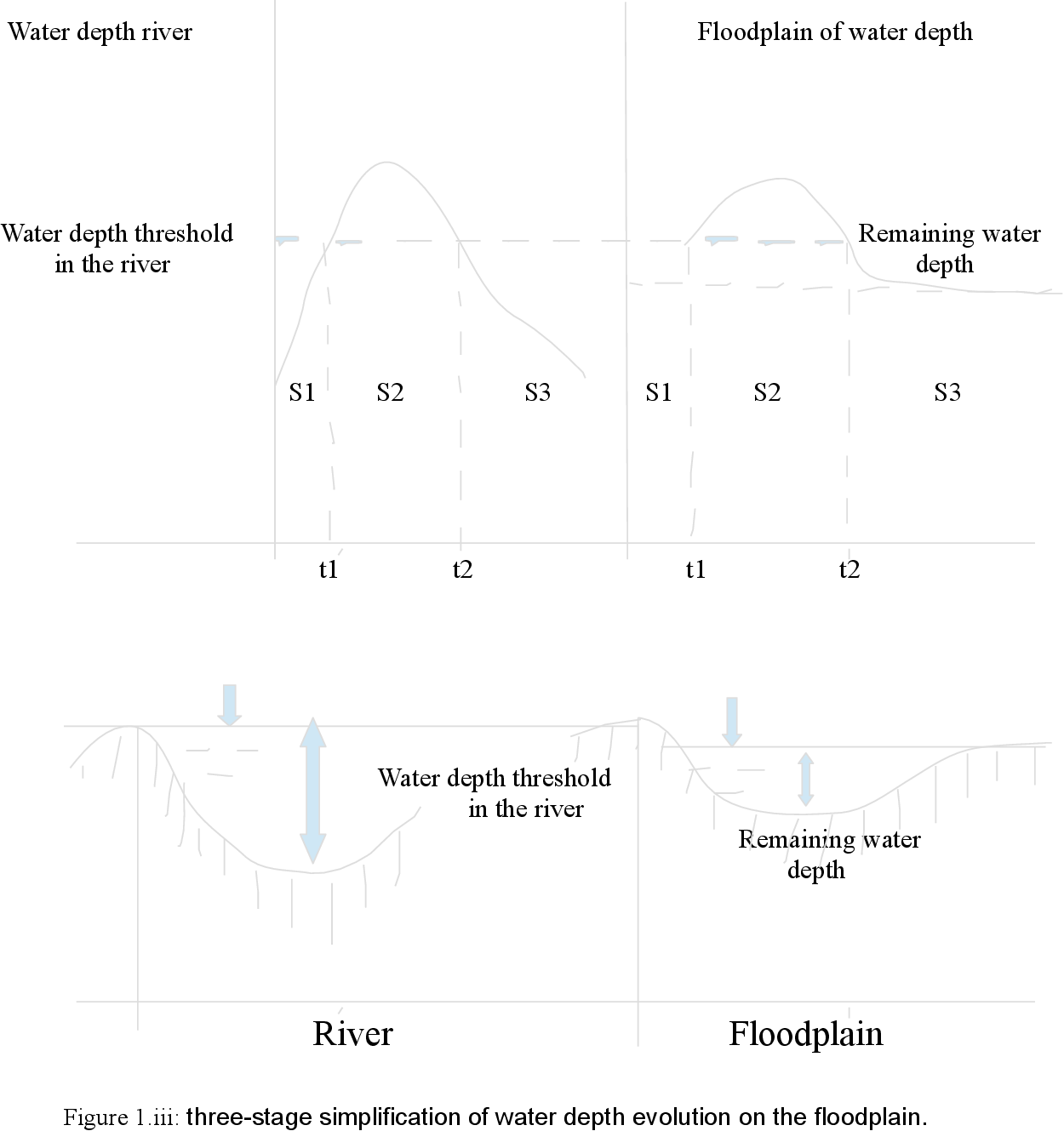,width=7cm} & \psfig{file=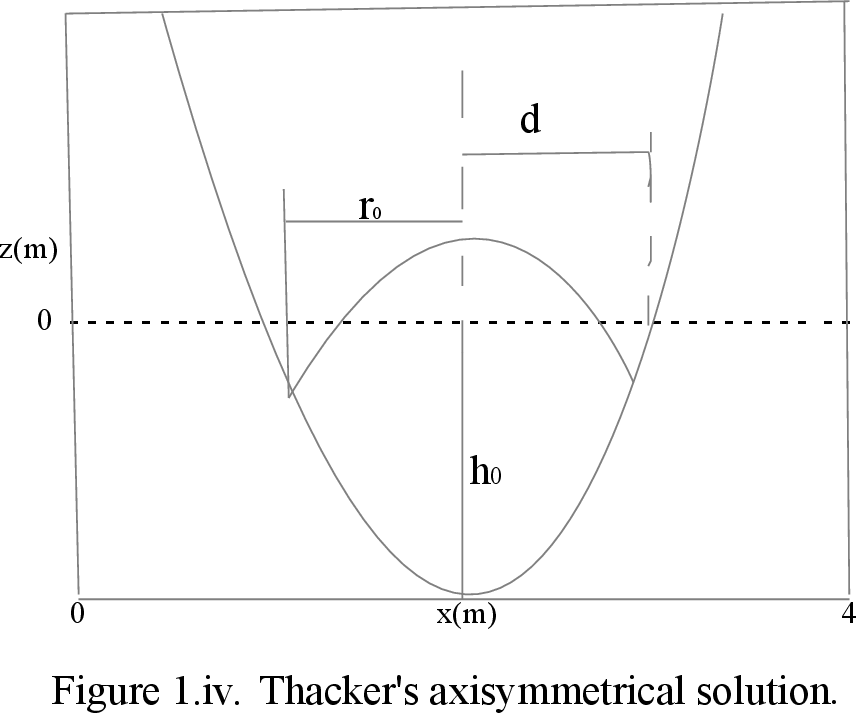,width=7cm}
         \end{tabular}
        \end{center}
        \caption{Dam-break, Water depth and Geometric configuration.}
        \label{fig1}
        \end{figure}

   This study considers a typical example of two-dimensional hydrodynamic model observed in the far north region of Cameroon. This region is faced to floods which have started in the second half of July $2024$. Due to the intensification and recurrence of heavy rains, series of floods are recorded in several localities in the region which have attained critical levels during the period from $11$ August-$25$ September $2024$, in the Mayo-Danay and Logone-et-Chari divisions (see Figure $\ref{fig2}$). The peak of the inundations was observed in $19$ September in which at least $67,323$ households (more than $365,060$ people including an estimated $124,120$ children under $5$ years old, have been severely affected by the floods and more than $30$ individuals died) \cite{2unicef}. In addition, approximately $82,509$ hectares of agriculture land are destroyed while some $5,278$ heads of livestock are lost \cite{unicef}. Regarding the health and education, $65$ health facilities across 15 out of the $19$ hospitals are impacted by the floods to various localities in region whereas $262$ schools are damaged, cut off or destroyed by inundations, affecting $103,906$ students and around $1,418$ teachers in the Logone-et-Chari and Mayo-Danay divisions \cite{4unicef,5unicef}. Furthermore, the floods also increased the risk to catch waterborne diseases such as malaria and cholera since the primary source of drinking water are either damaged, submerged or inaccessible. Latrines are also flooded and/or inaccessible to households \cite{unicef}. Although on $19$ September $2024$, the Cameroon government announced that the head of state has increased its support to victims to $1.9$ billion francs CFA (equivalent USD $\$3.1$ million), a broad range of persons who lost their homes are hosted by several families in the nearby locations and other are still sleeping in public buildings or out in the open, near their destroyed or damaged shelters. Additionally, $26,159$ displaced people including $5,633$ children, $658$ breastfeeding and $5413$ pregnant women, across eight sides are identified in the Yagoua(5), Maga(2) and Vele(1) sub-divisions in the Mayo-Danay division. New estimations indicate that $463,563$ persons will be affected by floods in the far north region in $2024$ \cite{6unicef,unicef}.\\

           \begin{figure}
         \begin{center}
         \begin{tabular}{c c}
         \psfig{file=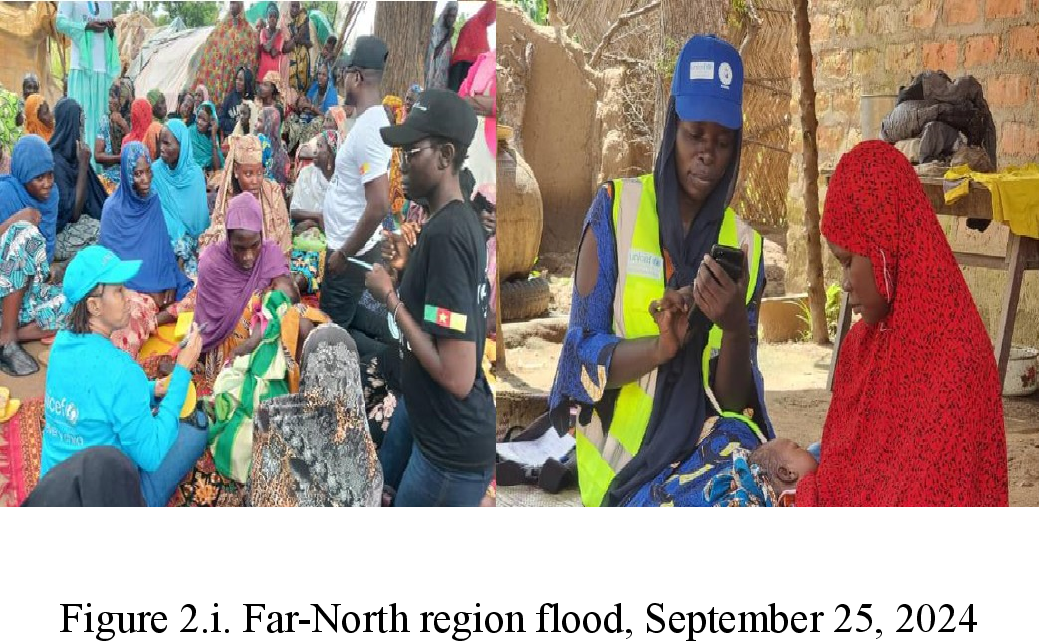,width=7cm} & \psfig{file=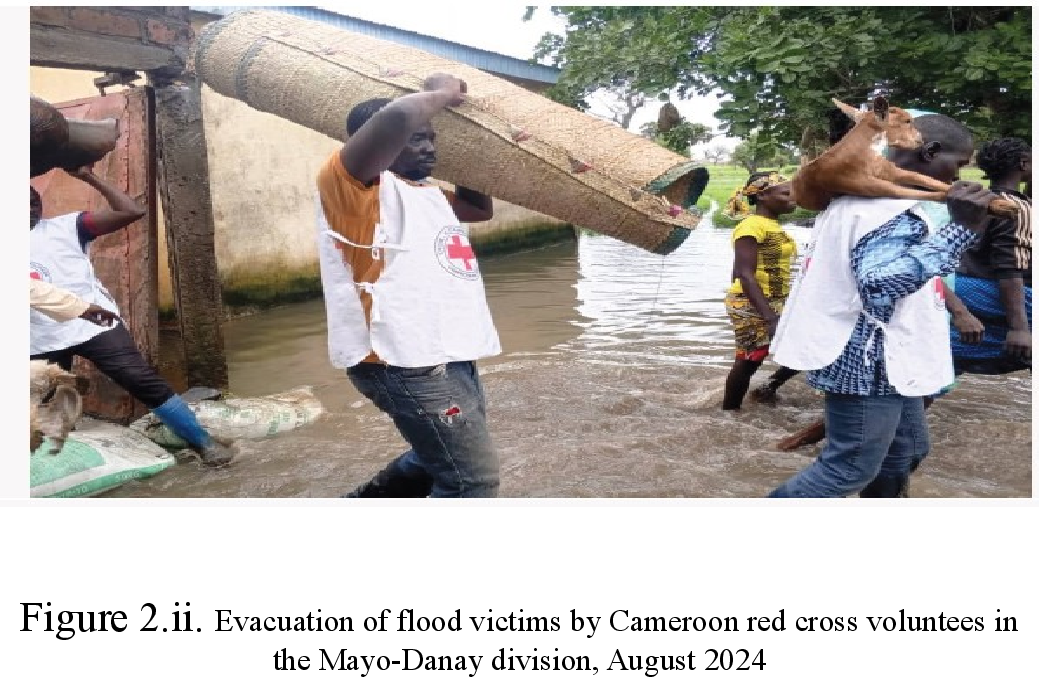,width=7cm}\\
         \psfig{file=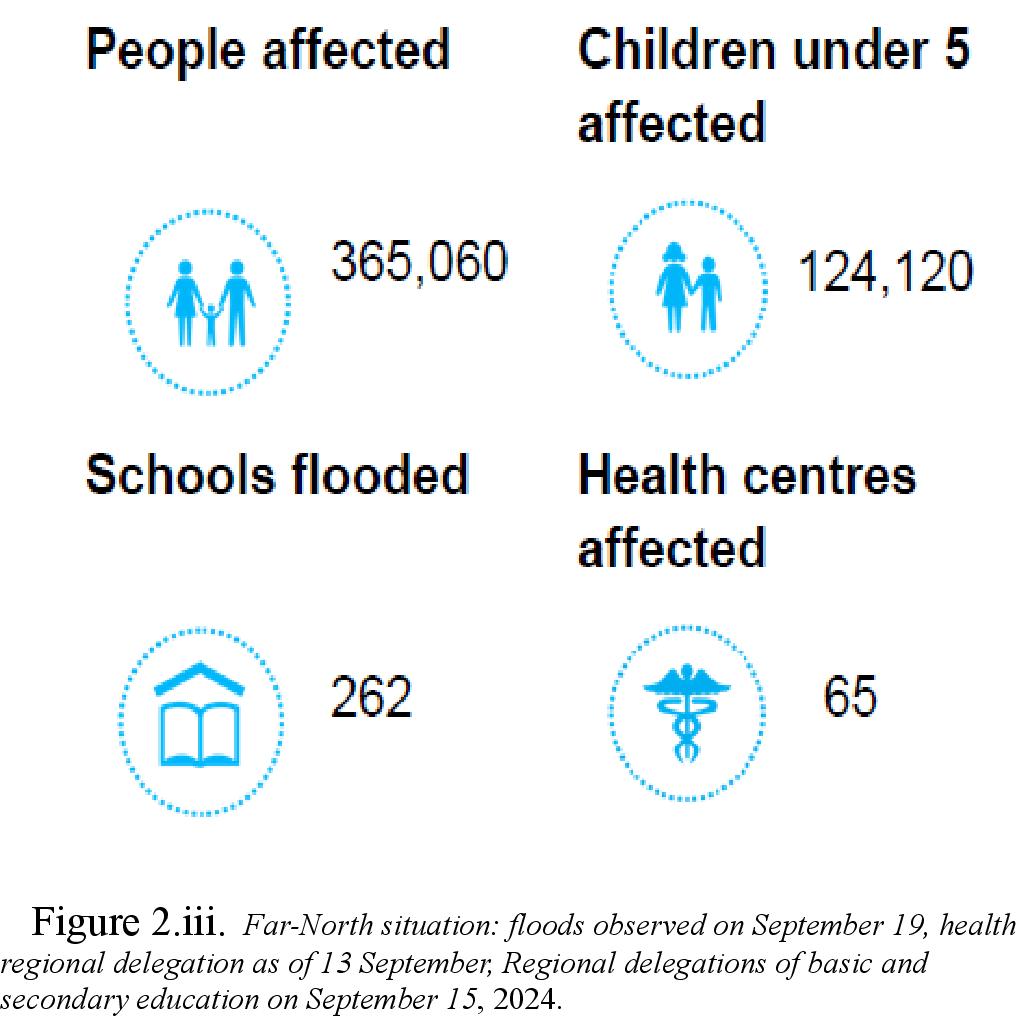,width=7cm} & \psfig{file=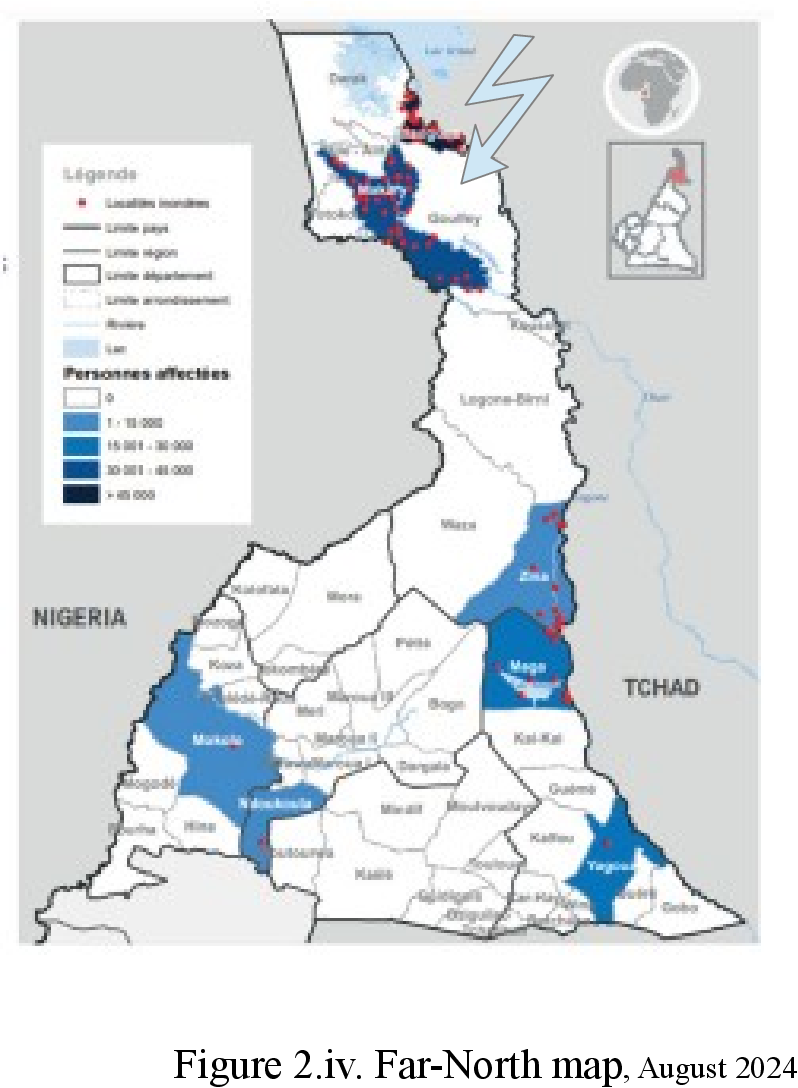,width=6cm}\\
         \psfig{file=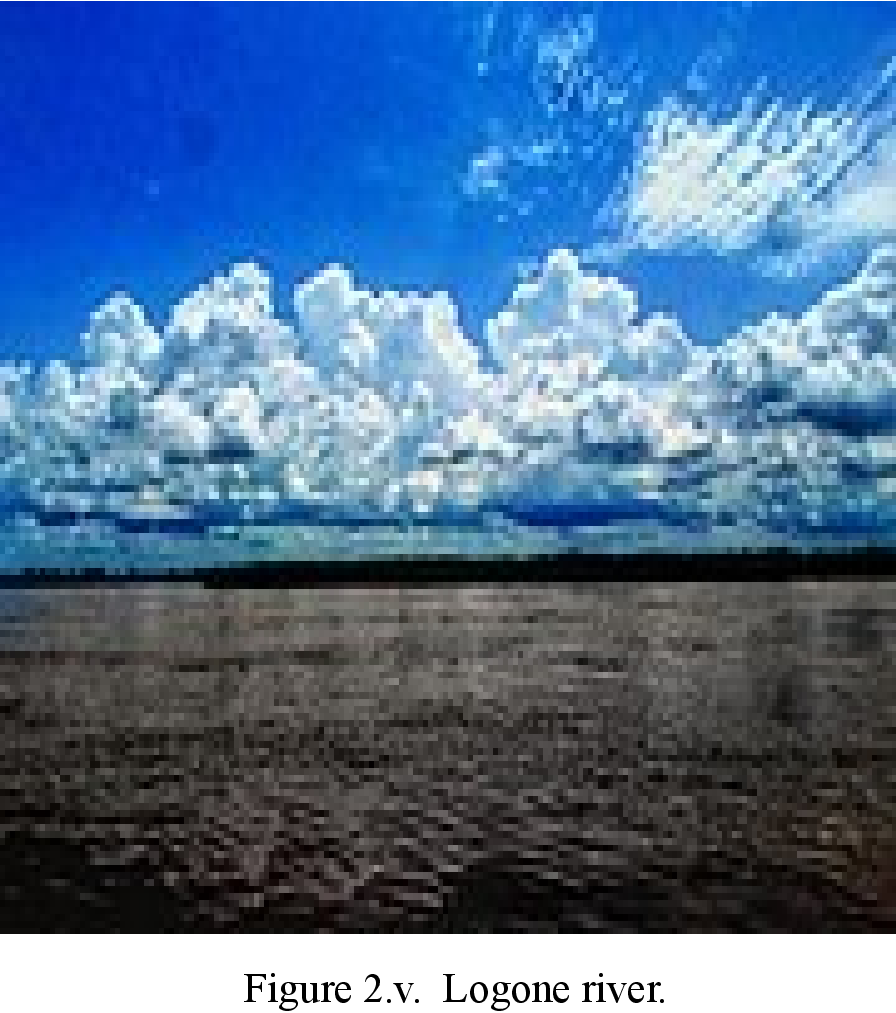,width=7cm} & \psfig{file=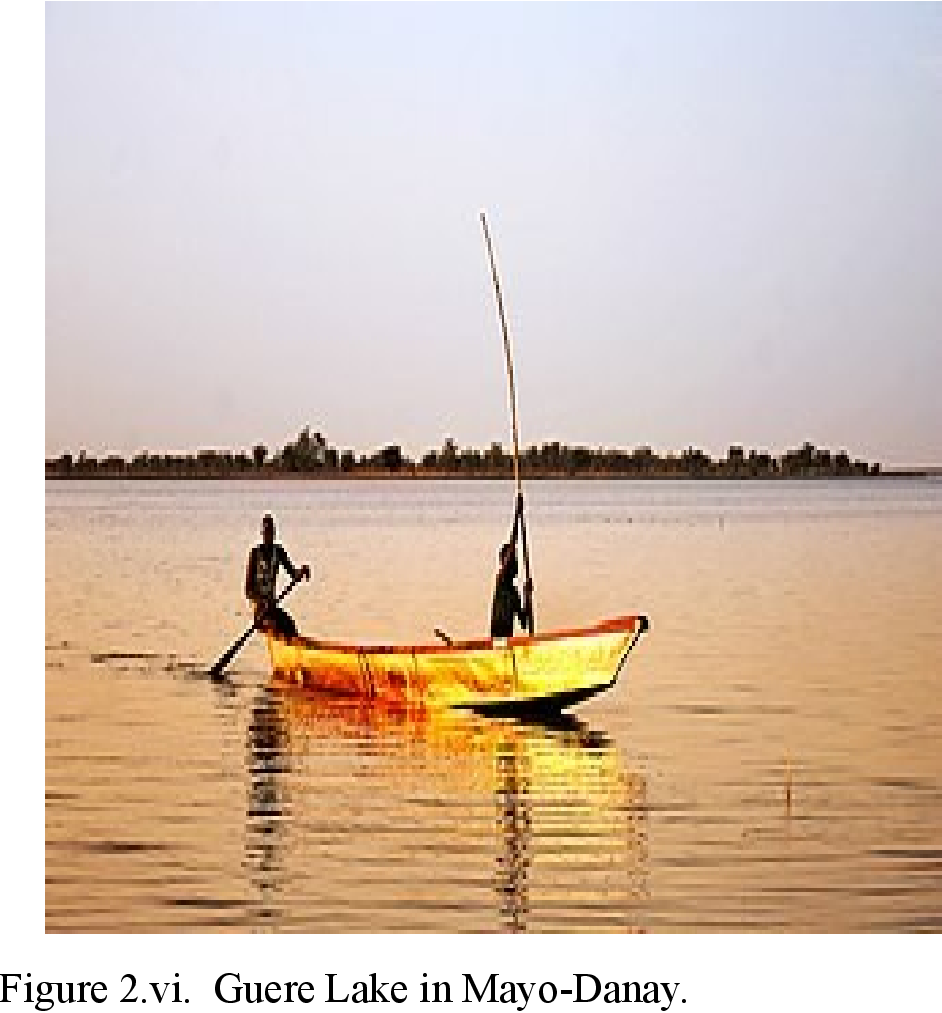,width=7cm}
         \end{tabular}
        \end{center}
        \caption{Situation of floods, Logone river and Guere lake in Mayo-Danay.}
        \label{fig2}
        \end{figure}

   The theoretical results developed in this paper deal with the two-dimensional hydrodynamic overland flow \cite{fr} which may be obtained from the Navier-Stokes equations by integrating over the depth using kinematic boundary conditions and making some assumptions such as: velocity is constant with depth, horizontal shear-stress together with the vertical velocity and acceleration are small and the pressure distribution is hydrostatic \cite{2fr,3fr}.
  \begin{equation}\label{1}
        \left\{
          \begin{array}{ll}
            \frac{\partial h}{\partial t}+\frac{\partial (hu)}{\partial x}+\frac{\partial (hv)}{\partial y}=0, & \hbox{}\\
            \text{\,}\\
            \frac{\partial (hu)}{\partial t}+\frac{\partial (hu^{2})}{\partial x}+\frac{\partial (huv)}{\partial y}=gh(S_{0_{x}}-\frac{\partial h}{\partial x}-S_{f_{x}}),&\hbox{}\\
            \text{\,}\\
            \frac{\partial (hv)}{\partial t}+\frac{\partial (huv)}{\partial x}+\frac{\partial (hv^{2})}{\partial y}=gh(S_{0_{y}}-\frac{\partial h}{\partial y}-S_{f_{y}}),&
            \hbox{}
          \end{array} \text{\,\,\,on\,\,\,} \Omega\times(0,\text{\,}T]
        \right.
      \end{equation}
      subjects to initial condition
      \begin{equation}\label{2}
       h(x,y,0)=\rho_{1}(x,y),\text{\,\,\,}u(x,y,0)=\rho_{2}(x,y),\text{\,\,\,}v(x,y,0)=\rho_{3}(x,y),\text{\,\,\,on\,\,\,}\overline{\Omega}=\Omega\cup\partial\Omega
      \end{equation}
      and boundary condition
      \begin{equation}\label{3b}
        h(x,y,t)=f_{1}(x,y,t),\text{\,\,\,}u(x,y,t)=f_{2}(x,y,t),\text{\,\,\,}v(x,y,t)=f_{3}(x,y,t),\text{\,\,\,for\,\,\,}(x,y,t)\in\partial\Omega\times[0,\text{\,}T].
      \end{equation}
        Here, $\Omega\subset \mathbb{R}^{2}$ is a boundary domain, $\partial\Omega$ denotes its boundary, $\frac{\partial}{\partial s}$ means the partial derivative with respect to the variable $s$ where $s\in\{t,x,y\}$, $h$ is the water depth, $u$ and $v$ are depth average velocities in the $x$-direction and $y$-direction, respectively, $S_{0_{x}}$ and $S_{0_{y}}$ denote the bed slops in the $x$-direction and $y$-direction, respectively, $S_{f_{x}}$ and $S_{f_{y}}$ represent the bottom frictions in the $x$-direction and $y$-direction, respectively, while $g$ means the gravitational acceleration and $T$ represents the time interval length. It's worth mentioning that the bottom friction should be estimated using the Manning's formulas
       \begin{equation}\label{3aa}
       S_{f_{x}}=\frac{\overline{n}^{2}(u^{\frac{3}{2}}+uv^{\frac{1}{2}})}{c_{0}^{2}h^{\frac{4}{3}}}\text{\,\,\,and\,\,\,}S_{f_{y}}=\frac{\overline{n}^{2}(v^{\frac{3}{2}}
       +vu^{\frac{1}{2}})}{c_{0}^{2}h^{\frac{4}{3}}},
      \end{equation}
      where $\overline{n}$ designates the Manning's roughness coefficient and $c_{0}$ is a dimensional constant. The first equation in relation $(\ref{1})$ derives from the conservation of mass over a control volume while both second and third equations in $(\ref{1})$ result from conservation of momentum in the $x$-direction and $y$-direction, respectively. The terms $gh(S_{0_{x}}-\frac{\partial h}{\partial x}-S_{f_{x}})$ and $gh(S_{0_{y}}-\frac{\partial h}{\partial y}-S_{f_{y}})$ in the momentum equations denote various quantities associated with conservation of momentum. Set
      \begin{equation}\label{3bb}
        \phi=h[1,u,v]^{t},\text{\,\,}E(\phi)=h[u,u^{2}+\frac{1}{2}gh,uv]^{t},\text{\,\,}F(\phi)=h[v,uv,v^{2}+\frac{1}{2}gh]^{t}\text{\,\,and\,\,}
        G(\phi)=gh[0, S_{0_{x}}-S_{f_{x}},S_{0_{y}}-S_{f_{y}}]^{t},
      \end{equation}
      where $w^{t}$ denotes the transpose of a vector $w\in\mathbb{R}^{3}$. The system of equations $(\ref{1})$ can be expressed in conservative form as
      \begin{equation}\label{4}
        \frac{\partial \phi}{\partial t}+\frac{\partial E(\phi)}{\partial x}+\frac{\partial F(\phi)}{\partial y}=G(\phi).
      \end{equation}
      In this new formulation, $\phi$ is a vector dealing with primitive variables: $h$, $u$ and $v$; $E(\phi)$ and $F(\phi)$ are two vectors written in flux form whereas $G(\phi)$ is the source term. However, the system of equations $(\ref{4})$ lie in the class of complex nonlinear partial/ordinary differential equations (PDEs/ODEs) that do not possess analytical solutions \cite{3en,4en,5en,6en,7en,8en,9en}. To overcome this issue, a large set of numerical methods such as: explicit difference schemes, implicit difference methods, explicit/implicit difference formulations, finite element methods, finite volume procedures, time-split approaches including time-split MacCormack rapid solver. The interested readers can refer to the works discussed in \cite{8db,10en,15db,16db,2en,17db,18db,17en,19db,20db,13en,91yzw,33yzw,14en,
      42yzw,28yzw,16en} and references therein. In \cite{21db,1en}, the authors established that one suitable technique to find efficient solutions of unsteady flow model dealing with the presence of inherent dissipation, discontinuity and stability, such as the problem given by the system of equations $(\ref{4})$ is the MacCormack scheme. This approach has provided time-accurate solution for aeroacoustic problems and fluid flows. Solve the one-dimensional shock tube along with two-dimensional acoustic scatting models using this computational method provide good results compared to the analytical solutions. Additionally, the MacCormack procedure is less time consuming and easy to implement than a wide set of numerical schemes mentioned above and should be suitable to give reliable results when applied to nonlinear unsteady flow models including the dam-break problems in the presence of discontinuity and strict gradient conditions \cite{21db,23db}. To construct the time-split MacCormack version, some authors \cite{23db} have modified the computational scheme analyzed in \cite{21db} into an implicit prefactorization method by splitting the derivative operators of a central compact scheme into one-side forward and backward difference operators. The prefactorization technique considers an implicit matrix that is decomposed into two independent upper and lower triangular matrices easier to convert. Hence, the one-side nature of the time-split MacCormack approach plays a crucial role on its efficiency especially when severe gradients are present \cite{23db,11en,26db,15en}. Though this scheme is less accurate than the more recent methods such as the one developed in this work, it is commonly used for engineering problems due to its greater simplicity. Furthermore, a major class of efficient numerical methods are based on explicit-implicit schemes and time-splitting methods, also called fractional techniques \cite{5yzw,12en,9yzw}. Most accurate and efficient time-split schemes are constructed according to either the physical components such as: velocity, pressure, density, energy of physical processes including convection, diffusion, reaction or dimension (for example, see the methods analyzed in \cite{10yzw,11yzw,12yzw}). Although the approaches may suffer from two disadvantages: boundary conditions corresponding to split equations and splitting error in the composite algorithms, high-order fractional steps procedures can be developed to reduce the splitting errors in the formulations discussed in \cite{10yzw,11yzw,27yzw,29yzw}, while intermediate boundary conditions should be obtained from the split equations as suggested in \cite{7yzw,8yzw}.\\

      However, high-order time-split methods are crucial tools in the integration of nonlinear systems that possess homoclinic orbits in the geometry of their phase space which cause a big challenge in the numerical integration \cite{13yzw,19yzw}. This paper develops a spatial fourth-order and temporal second-order symmetric time-split explicit/implicit computational approach for simulating the two-dimensional initial-boundary value problem $(\ref{1})$-$(\ref{3b})$. The new technique is constructed by exploiting the splitting formulas studied in \cite{12yzw} together with the idea used by MacCormack to develop the time-split rapid solver. The new symmetric fractional steps explicit/implicit technique is composed of three stages. In the first step, the one-dimensional difference operator in the $x$-direction computes explicitly while in the second stage, the one-dimensional difference operator in the $y$-direction calculates implicitly. Finally, in the third step the one-dimensional difference operator in the $x$-direction computes explicitly. This procedure takes advantage to be less time consuming and to avoid the severe time steps restriction for stability along with the inversion of block matrices while preserving the second-order accurate in time. Indeed, the errors increased at the first and third stages are balanced by the ones decreased at the second step, so that the stability might be maintained since the proposed computational scheme is symmetric and the sum of the time-steps of each difference operators in the composition is equal. In addition, the constructed time-split explicit/implicit method is efficient and more faster than a broad range of numerical schemes \cite{9yzw,26db,8db,19db}, analyzed in the literature for solving the two-dimensional hydrodynamic problem $(\ref{1})$ with appropriate initial and boundary conditions $(\ref{2})$-$(\ref{3b})$.\\

      The remainder of the paper is organized as follows. Section $\ref{sec2}$ deals with a detailed description of the symmetric time-split explicit/implicit approach in an approximate solution of the dam-break problem $(\ref{4})$, subjected to suitable initial condition $(\ref{2})$ and boundary one $(\ref{3b})$. The Courant-Friedrichs-Lewy (CFL) requirement for necessary condition of stability of explicit numerical methods applied to linear hyperbolic PDEs together with a deep analysis of the time step restriction for stability of the proposed time-split explicit/implicit technique are provided in Section $\ref{sec3}$. In Section $\ref{sec4}$, a large set of numerical experiments are carried out and discussed to confirm the theoretical studies. Furthermore, the numerical tests also consider the floods observed in the far north region of Cameroon from the second half in July up to the second half in October $2024$. Finally, the general conclusions and our future works are presented in Section $\ref{sec5}$.

     \section{Construction of the new approach}\label{sec2}
     In this section, we develop a linearized time-split explicit/implicit technique for solving the two-dimensional shallow water equations $(\ref{1})$ subjects to initial-boundary conditions $(\ref{2})$-$(\ref{3b})$.\\

     To construct the method, the two-dimensional time-dependent problem $(\ref{4})$ is converted into two one-dimensional subproblems by the use of locally-one dimensional (LOD) time-splitting procedure:
     \begin{equation}\label{5}
      \frac{\partial \phi}{\partial t}=-\frac{\partial E(\phi)}{\partial x},
      \end{equation}
      \begin{equation}\label{6}
       \frac{\partial \phi}{\partial t}=G(\phi)-\frac{\partial F(\phi)}{\partial y}.
      \end{equation}
      We introduce the one-dimensional explicit and implicit difference operators $\mathcal{P}_{1}(k_{x})$ and $\mathcal{P}_{2}(k_{y})$ associated with equations $(\ref{5})$ and $(\ref{6})$, respectively.
      \begin{equation}\label{7}
       \phi^{*}=\mathcal{P}_{1}(k_{x})\phi^{n},
      \end{equation}
      \begin{equation}\label{8}
       \phi^{**}=\mathcal{P}_{2}(k_{y})\phi^{*},
      \end{equation}
      where the asterisks "*" and "**" are symbols of convenience and they denote the intermediate time levels ($nk<t_{*}<t_{**}\leq(n+1)k$), $\mathcal{P}_{2}(k_{y})\phi^{*}=\frac{1}{2}[\overline{\mathcal{P}}_{2}(k_{y})\phi^{*}+\overline{\mathcal{P}}_{2}(k_{y})\overline{\phi}^{**}]$ and $\phi^{l}=\phi(t_{l})$. Here, $\overline{\mathcal{P}}_{2}(k_{y})$ represents an explicit difference scheme. Setting: $k_{x}=\frac{k}{2}$ and $k_{y}=k$, where $k$ designates the time step, the new time-split explicit/implicit will derive from the following formula
      \begin{equation}\label{9}
       \phi^{n+1}=\mathcal{P}(k)\phi^{n}=\mathcal{P}_{1}(k/2)\mathcal{P}_{2}(k)\mathcal{P}_{1}(k/2)\phi^{n},
      \end{equation}
       The motivation to take $k_{x}=\frac{k}{2}$, $k_{y}=k$ and to define the operator $\mathcal{P}(k)=\mathcal{P}_{1}(k/2)\circ\mathcal{P}_{2}(k)\circ\mathcal{P}_{1}(k/2)$, is to ensure the consistency and temporal second-order convergence of the developed technique, where "$\circ$" means the composite operator. Indeed, it's well known in the literature that a sequence of operators is consistent if the sum of the time-steps for each of the operators is equal and it is temporal second-order convergent when the sequence is symmetric.\\

      Let $N$, $M_{x}$ and $M_{y}$ be three positive integers. Set $k=\frac{T}{N}$ be the time step and $\Delta x$ and $\Delta y$, be the space steps in the $x$-direction and $y$-direction, respectively. The domain $\Omega$ is discretized into $(M_{x}+1)(M_{y}+1)$ non-overlapping and uniform quadrilaterals $T_{lp}$, $l=0,1,...,M_{x}$ and $p=0,1,...,M_{y}$, with center $(x_{l},y_{p})$ and dimensions $\Delta x\times\Delta y$. Thus, $T_{lp}$ can be represented as: $T_{lp}=[x_{l}-\frac{\Delta x}{2},\text{\,}x_{l}+\frac{\Delta x}{2}]\times[y_{p}-\frac{\Delta y}{2},\text{\,}y_{p}+\frac{\Delta y}{2}]$. Furthermore, it is not difficult to observe that $T_{lp}$ can be mapped into $[-1,\text{\,}1]^{2}$, using the following transformation
      \begin{equation*}
      \left.
          \begin{array}{ll}
            \psi:\text{\,\,}T_{lp}\rightarrow [-1,\text{\,}1]^{2} & \hbox{}\\
            (x,y)\mapsto 2\left(\frac{x-x_{l}}{\Delta x},\frac{y-y_{p}}{\Delta y}\right).&\hbox{}
          \end{array}
        \right.
      \end{equation*}
      Utilizing this fact, we assume that $(0,0)\in\Omega$. Set $x_{l}=l\Delta x$ and $y_{p}=p\Delta y$. In the following we use notation $w_{lp}^{n}=w(x_{l},y_{p},t_{n})$. Additionally, we denote both computed solution and the exact one of problem $(\ref{4})$ with initial and boundary conditions $(\ref{2})$ and $(\ref{3b})$, at the discrete point $(x_{l},y_{p},t_{n})$ by $\overline{\phi}_{lp}^{n}$ and $\phi_{lp}^{n}$, respectively, while $W=\{w_{lp}^{n}:\text{\,}0\leq n\leq N;
      \text{\,}0\leq l\leq M_{x};\text{\,}0\leq p\leq M_{y}\}$, represents the space of grid functions defined on $\Omega_{\Delta xy}\times\Omega_{k}$, where $\Omega_{k}=\{t_{n},\text{\,}n=0,1,...,N\}$ and $\Omega_{\Delta xy}=\Omega\cap\overline{\Omega}_{\Delta xy}$, with $\overline{\Omega}_{\Delta xy}=\{(x_{l},y_{p}):\text{\,}l=0,1,...,M_{x};\text{\,}p=0,1,...M_{y}\}$.\\

      We introduce the following centered difference operators of order two and order four along with the forward and backward difference operators of order three:
      \begin{equation*}
      \delta_{x}^{2}w_{lp}^{n}=\frac{w_{l+1,p}^{n}-w_{l-1,p}^{n}}{2\Delta x};\text{\,\,\,}\delta_{y}^{2}w_{lp}^{n}=\frac{w_{l,p+1}^{n}-w_{l,p-1}^{n}}{2\Delta y};\text{\,\,\,}
      \delta_{x}^{4}w_{lp}^{n}=\frac{1}{12\Delta x}[-w_{l+2,p}^{n}+8(w_{l+1,p}^{n}-w_{l-1,p}^{n})+w_{l-2,p}^{n}];\text{\,\,\,}
      \end{equation*}
      \begin{equation*}
      \delta_{y}^{4}w_{lp}^{n}=\frac{1}{12\Delta y}[-w_{l,p+2}^{n}+8(w_{l,p+1}^{n}-w_{l,p-1}^{n})+w_{l,p-2}^{n}];\text{\,\,\,}\delta_{x}^{3+}w_{lp}^{n}=\frac{1}{6\Delta x}[-w_{l+2,p}^{n}+6w_{l+1,p}^{n}-3w_{l,p}^{n}-2w_{l-1,p}^{n}];
      \end{equation*}
      \begin{equation*}
      \delta_{y}^{3+}w_{lp}^{n}=\frac{1}{6\Delta y}[-w_{l,p+2}^{n}+6w_{l,p+1}^{n}-3w_{l,p}^{n}-2w_{l,p-1}^{n}];\text{\,\,\,}\delta_{x}^{3-}w_{lp}^{n}=\frac{1}{6\Delta x}[2w_{l+1,p}^{n}+3w_{l,p}^{n}-6w_{l-1,p}^{n}+w_{l-2,p}^{n}];
      \end{equation*}
      \begin{equation}\label{9a}
      \delta_{y}^{3-}w_{lp}^{n}=\frac{1}{6\Delta y}[2w_{l,p+1}^{n}+3w_{l,p}^{n}-6w_{l,p-1}^{n}+w_{l,p-2}^{n}],
      \end{equation}
      for $l=2,3,...,M_{x}-2$ and $p=2,3,...,M_{y}-2$.
      \begin{remark}
      It's worth noticing to mention that the linear operators $\delta_{z}^{2}$, $\delta_{z}^{3+}$, $\delta_{z}^{3-}$ and $\delta_{z}^{4}$, satisfy
      \begin{equation}\label{10aa}
      \delta_{z}^{2}(w\delta_{z}^{3\mp})\psi_{lp}=\frac{1}{2\Delta z}[w_{l+1,p}\delta_{z}^{3-}\psi_{l+1,p}-w_{l-1,p}\delta_{z}^{3+}\psi_{l-1,p}]\text{\,\,\,and\,\,\,}\delta_{z}^{4}\psi_{lp}=\frac{1}{2}[\delta_{z}^{3+}\psi_{lp}+\delta_{z}^{3-}
      \psi_{lp}],
      \end{equation}
      where $z=x,y$ and for every $w,\psi\in L^{2}(\Omega)$.
      \end{remark}
      The spaces $L^{2}(\Omega)$, $\mathbb{R}^{M_{x}-3}$, $\mathbb{R}^{M_{y}-3}$, and $\left(\mathbb{R}^{M_{x}-3}\right)^{3}$, are equipped with the following scalar products and norms
      \begin{equation}\label{10}
      \left(w,q\right)_{0}=\Delta_{x}\Delta_{y}\underset{l=2}{\overset{M_{x}-2}\sum}\underset{p=2}{\overset{M_{y}-2}\sum}w_{lp}q_{lp},\text{\,\,\,}\|w\|_{0}=\sqrt{\left(w,w\right)_{0}},
      \text{\,\,\,\,\,\,\,}\forall w,p\in L^{2}(\Omega),
      \end{equation}
      \begin{equation}\label{11a}
      \left(d,z\right)_{\sim}=\underset{l=2}{\overset{M_{x}-2}\sum}d_{l}z_{l},\text{\,\,\,\,\,}\|z\|_{\sim}=\sqrt{\left(z,z\right)_{\simeq}},\text{\,\,\,\,\,}
      \left(\overline{d},\overline{z}\right)_{\simeq}=\underset{l=2}{\overset{M_{y}-2}\sum}
      \overline{d}_{l}\overline{z}_{l},\text{\,\,\,\,\,\,}\|\overline{z}\|_{\sim}=\sqrt{\left(\overline{z},\overline{z}\right)_{\simeq}}.
      \end{equation}
       for every $d,z\in \mathbb{R}^{M_{x}-3},$ and all $\overline{d},\overline{z}\in \mathbb{R}^{M_{y}-3}$.
       \begin{equation}\label{10a}
      \||W|\|_{0,\sim}=(\|w_{1}\|_{0}^{2}+\|w_{2}\|_{0}^{2}+\|w_{3}\|_{0}^{2})^{\frac{1}{2}},\text{\,\,\,\,\,\,}\forall W=(w_{1},w_{2},w_{3})\in \left(\mathbb{R}^{M_{x}-3}\right)^{3}.
      \end{equation}
       The matrix norm, $\||\cdot|\|$, associated with the norm, $\|\cdot\|_{\sim}$, is defined as
       \begin{equation}\label{12a}
        \||A|\|=\underset{0\neq z\in\mathbb{R}^{M_{x}-3}}{\max}\frac{\|Az\|_{\sim}}{\|z\|_{\sim}},\text{\,\,\,\,\,\,}\forall A\in \mathcal{M}_{M_{x}-3}(\mathbb{R}),
      \end{equation}
      where $\mathcal{M}_{M_{x}-3}(\mathbb{R})$ denotes the vector space of $(M_{x}-3)\times(M_{x}-3)$ matrices. Additionally, the Sobolev space $L^{\infty}(0,T;L^{2})$ is endowed with the norm
      \begin{equation}\label{11}
      \||\overline{w}|\|_{0,\infty}=\underset{0\leq n\leq N}{\max}\|\overline{w}^{n}\|_{0},\text{\,\,\,\,\,\,}\forall \overline{w}\in L^{\infty}(0,T;L^{2}).
      \end{equation}
      The following Lemma plays a crucial role in the development of the new algorithm.

      \begin{lemma}\label{l1}
      Suppose $w\in \mathcal{H}^{5}(\Omega)$, space of functions having continuous partial derivatives up to order four. Using relation $(\ref{9a})$, the following approximations are satisfied
      \begin{equation}\label{11aa}
      \frac{\partial w}{\partial z}=\delta^{2}_{z}w+O(\Delta z^{2}),\text{\,\,\,}\frac{\partial w}{\partial z}=\delta^{3+}_{z}w+O(\Delta z^{3}),\text{\,\,\,}
      \frac{\partial w}{\partial z}=\delta^{3-}_{z}w+O(\Delta z^{3}),\text{\,\,\,}\frac{\partial w}{\partial z}=\delta^{4}_{z}w+O(\Delta z^{4}),
      \end{equation}
      where $z=x,y$.

      \begin{proof}
      The proof of the first and last approximations can be found in \cite{5en}. Expanding the Taylor series for $w$ with space step $\Delta z$ at the discrete point $(x_{l},y_{p})$, up to order four, using both forward and backward difference formulas and performing direct calculations to get the second and third equations.
      \end{proof}
      \end{lemma}

      Now, the application of the Taylor series for the function $\phi$ with time step $\frac{k}{2}$ at the discrete point $(x_{l},y_{p},t_{n})$ using forward difference formula provides
      \begin{equation*}
      \phi^{*}_{lp}=\phi^{n}_{lp}+\frac{k}{2}\phi^{n}_{t,lp}+\frac{k^{2}}{8}\phi^{n}_{2t,lp}+O(k^{3}).
      \end{equation*}
      Utilizing equation $(\ref{5})$, this becomes
      \begin{equation}\label{12}
     \phi^{*}_{lp}=\phi^{n}_{lp}-\frac{k}{2}E(\phi)_{x,lp}^{n}+\frac{k^{2}}{8}\phi^{n}_{2t,lp}+O(k^{3}),
      \end{equation}
      and
      \begin{equation*}
      \phi^{n}_{2t,lp}=\frac{\partial}{\partial t}(\phi_{t})^{n}_{lp}=-\frac{\partial}{\partial t}(E(\phi)_{x})^{n}_{lp}=-\frac{\partial}{\partial x}(E(\phi)_{t})^{n}_{lp}=
      -\frac{\partial}{\partial x}(\nabla_{\phi}E(\phi)\cdot\phi_{t})^{n}_{lp}=\frac{\partial}{\partial x}\left(\nabla_{\phi}E(\phi)\cdot E(\phi)_{x}\right)^{n}_{lp},
      \end{equation*}
      where $\nabla_{\phi}E(\phi)$ designates the Jacobian matrix of the vector $E(\phi)$ (which is a square matrix) and $\nabla_{\phi}Q\cdot w$ means the matrix-vector multiplication. Substituting this into equation $(\ref{12})$ to get
      \begin{equation}\label{13}
     \phi^{*}_{lp}=\phi^{n}_{lp}-\frac{k}{2}E(\phi)_{x,lp}^{n}+\frac{k^{2}}{8}\left(\nabla_{\phi}E(\phi)\cdot E(\phi)_{x}\right)^{n}_{lp}+O(k^{3}),
      \end{equation}
      But the approximation of the terms $E(\phi)_{x,lp}^{n}$ and $\left(\nabla_{\phi}E(\phi)\cdot E(\phi)_{x}\right)^{n}_{x,lp}$, utilizing Lemma $\ref{l1}$ together with the centered, forward and backward difference operators: $\delta^{4}_{x}$, $\delta^{2}_{x}$, $\delta^{3+}_{x}$ and  $\delta^{3-}_{x}$, defined in relation $(\ref{9a})$, gives
      \begin{equation}\label{14}
      E(\phi)_{x,lp}^{n}=\delta^{4}_{x}E(\phi)_{lp}^{n}+O(\Delta x^{4}),
      \end{equation}
      \begin{equation*}
      \left(\nabla_{\phi}E(\phi)\cdot E(\phi)_{x}\right)^{n}_{x,lp}+O(\Delta x^{2})=\frac{1}{2\Delta x}[\nabla_{\phi}E(\phi^{n}_{l+1,p})\cdot E(\phi^{n}_{x,l+1,p})-
      \nabla_{\phi}E(\phi^{n}_{l-1,p})\cdot E(\phi)^{n}_{x,l-1,p}]+O(\Delta x^{2})
      \end{equation*}
      \begin{equation}\label{15}
      =\frac{1}{2\Delta x}[\nabla_{\phi}E(\phi^{n}_{l+1,p})\cdot\delta^{3-}_{x}E(\phi^{n}_{l+1,p})-\nabla_{\phi}E(\phi^{n}_{l-1,p})\cdot \delta^{3+}_{x}E(\phi)^{n}_{l-1,p}]
      +O(\Delta x^{2}).
      \end{equation}
      It's important to remind that the use of backward and forward differences ($\delta^{3-}_{x}$ and $\delta^{3+}_{x}$, respectively) in the last equality eliminate any bias due to the one-sided differencing. Substituting approximations $(\ref{14})$ and $(\ref{15})$ into equation $(\ref{13})$ and rearranging terms result in
      \begin{equation*}
     \phi^{*}_{lp}=\phi^{n}_{lp}-\frac{k}{2}\delta^{4}_{x}E(\phi)_{lp}^{n}+\frac{k^{2}}{16\Delta x}\left(\nabla_{\phi}E(\phi^{n}_{l+1,p})\cdot\delta^{3-}_{x}E(\phi^{n}_{l+1,p})-
     \nabla_{\phi}E(\phi^{n}_{l-1,p})\cdot \delta^{3+}_{x}E(\phi)^{n}_{l-1,p}\right)
      \end{equation*}
      \begin{equation}\label{18}
     +O(k^{3}+k^{2}\Delta x^{2}+k\Delta x^{4})=\phi^{n}_{lp}-\frac{k}{2}\delta^{4}_{x}E(\phi)^{n}_{lp}+\frac{k^{2}}{8}\delta^{2}_{x}\left(
     \nabla_{\phi}E(\phi)\cdot\delta^{3\mp}_{x}E(\phi)\right)^{n}_{lp}+O(k^{3}+k^{2}\Delta x^{2}+k\Delta x^{4}),
      \end{equation}
      where $\delta^{2}_{x}\left(w\delta^{3\mp}_{x}q\right)^{n}_{lp}$ is defined in equation $(\ref{10aa})$. Tracking the error term $O(k^{3}+k^{2}\Delta x^{2}+k\Delta x^{4})$ and replacing the exact solution $\phi=(h,hu,hv)^{t}$ with the approximate one $\overline{\phi}=(\overline{h},\overline{h}\overline{u},\overline{h}\overline{v})^{t}$, this provides the expression of the nonlinear difference operator $\mathcal{P}_{1}(k/2)$, that is,
       \begin{equation}\label{19}
     \overline{\phi}^{*}_{lp}=\mathcal{P}_{1}(k/2)\overline{\phi}^{n}_{lp},\text{\,\,\,for\,\,\,}l=2,3,...,M_{x}-2,\text{\,\,\,}p=0,1,...,M_{y},
      \end{equation}
      where $\mathcal{P}_{1}(k/2)$ is defined as
      \begin{equation}\label{20}
      \mathcal{P}_{1}(k/2)=\mathcal{I}-\frac{k}{2}\delta^{4}_{x}E(\cdot)+\frac{k^{2}}{8}\delta^{2}_{x}(\nabla_{\phi}E(\cdot)\cdot\delta^{3\mp}_{x}E(\cdot)).
      \end{equation}
      Here, $\mathcal{I}$ represents the identity operator and the vectors: $\overline{\phi}^{n}_{lp}$, $E(\overline{\phi})^{n}_{lp}$ and matrix $\nabla_{\phi}E(\overline{\phi})^{n}_{lp}$ are given by
      \begin{equation}\label{21}
      \overline{\phi}^{n}_{lp}=\begin{bmatrix}
                                 \overline{h}^{n}_{lp} \\
                                 \text{\,}\\
                                 \overline{h}^{n}_{lp}\overline{u}^{n}_{lp} \\
                                 \text{\,}\\
                                 \overline{h}^{n}_{lp}\overline{v}^{n}_{lp} \\
                               \end{bmatrix},\text{\,\,\,}E(\overline{\phi})^{n}_{lp}=\begin{bmatrix}
                                 \overline{h}^{n}_{lp}\overline{u}^{n}_{lp} \\
                                 \text{\,}\\
                                 \overline{h}^{n}_{lp}\overline{u}^{2,n}_{lp}+\frac{1}{2}g\overline{h}^{2,n}_{lp} \\
                                 \text{\,}\\
                                 \overline{h}^{n}_{lp}\overline{u}^{n}_{lp}\overline{v}^{n}_{lp}\\
                               \end{bmatrix},\text{\,\,\,}\nabla_{\phi}E(\overline{\phi})^{n}_{lp}=\begin{bmatrix}
                                                                                                     \overline{u}^{n}_{lp} & \overline{h}^{n}_{lp} & 0 \\
                                                                                                     \text{\,}\\
                                                                                                     \overline{u}^{2,n}_{lp}+g\overline{h}^{n}_{lp} & 2\overline{h}^{n}_{lp}\overline{u}^{n}_{lp} & 0 \\
                                                                                                     \text{\,}\\
                                                                                                     \overline{u}^{n}_{lp}\overline{v}^{n}_{lp} & \overline{h}^{n}_{lp}\overline{v}^{n}_{lp} & \overline{h}^{n}_{lp}\overline{u}^{n}_{lp} \\
                                                                                                   \end{bmatrix}.
      \end{equation}
      Plugging equations $(\ref{18})$ and $(\ref{20})$, it is easy to see that
      \begin{equation}\label{22}
      \frac{1}{k}(\overline{\phi}^{*}_{lp}-\overline{\phi}^{n}_{lp})=\frac{1}{k}[\mathcal{P}_{1}(k/2)-\mathcal{I}]\overline{\phi}^{n}_{lp}+O(k^{2}+k\Delta x^{2}+\Delta x^{4}).
      \end{equation}
      Since $k\Delta x^{2}\leq\frac{1}{2}(k^{2}+\Delta x^{4})$, this fact together with approximation $(\ref{22})$ suggest that the difference scheme $(\ref{19})$ should converge with order $O(k^{2}+\Delta x^{4})$, whenever it is stable.\\

      Now, we should construct implicit formulation defined by the nonlinear difference operator $\mathcal{P}_{2}(k)$. Expanding the Taylor series for the function $\phi$ with step size $k$ at the discrete point $(x_{l},y_{p},t_{n})$ utilizing both forward and backward difference schemes yields
      \begin{equation*}
      \phi^{**}_{lp}=\phi^{*}_{lp}+k\phi^{*}_{t,lp}+\frac{k^{2}}{2}\phi^{*}_{2t,lp}+O(k^{3}),
      \end{equation*}
      \begin{equation*}
     \phi^{*}_{lp}=\phi^{**}_{lp}-k\phi_{t,lp}^{**}+\frac{k^{2}}{2}\phi^{**}_{2t,lp}+O(k^{3}),
      \end{equation*}
      where $t_{n}<t_{*}<t_{**}\leq t_{n+1}$. Subtracting the second equation from the first one and rearranging terms result in
      \begin{equation}\label{23}
      2(\phi^{**}_{lp}-\phi^{*}_{lp})=k(\phi_{t,lp}^{**}+\phi_{t,lp}^{*})+\frac{k^{2}}{2}(\phi^{*}_{2t,lp}-\phi_{2t,lp}^{**})+O(k^{3}).
      \end{equation}
      By the use of Mean-value theorem, it holds
      \begin{equation}\label{24}
      \phi^{*}_{2t,lp}-\phi_{2t,lp}^{**}=(t_{*}-t_{**})\phi_{3t,lp}(t_{3*}),
      \end{equation}
      where $t_{n}<t_{*}\leq t_{3*}\leq t_{**}\leq t_{n+1}$, so $t_{*}-t_{**}=O(k)$. This fact combined with equations $(\ref{23})$-$(\ref{24})$ give
       \begin{equation*}
      2(\phi^{**}_{lp}-\phi^{*}_{lp})=k(\phi_{t,lp}^{**}+\phi_{t,lp}^{*})+O(k^{3}),
      \end{equation*}
      which is equivalent to
       \begin{equation*}
       \phi^{**}_{lp}=\phi^{*}_{lp}+\frac{k}{2}(\phi_{t,lp}^{**}+\phi_{t,lp}^{*})+O(k^{3}).
      \end{equation*}
      Using equation $(\ref{6})$, this becomes
       \begin{equation}\label{25}
       \phi^{**}_{lp}=\phi^{*}_{lp}+\frac{k}{2}[G(\phi)_{lp}^{**}+G(\phi)_{lp}^{*}-(F_{y}(\phi)_{lp}^{**}+F_{y}(\phi)_{lp}^{*})]+O(k^{3}).
      \end{equation}
      The spatial approximation of the terms $F_{y}(\phi)_{lp}^{*}$ and $F_{y}(\phi)_{lp}^{**}$, using Lemma $\ref{l1}$ together with the centered difference formulation, $\delta^{4}_{y}$, defined in relation $(\ref{9a})$ provides
       \begin{equation*}
       F_{y}(\phi)^{*}_{lp}=\delta^{4}_{y}F(\phi)^{*}_{lp}+O(\Delta y^{4})\text{\,\,\,and\,\,\,}F_{y}(\phi)^{**}_{lp}=\delta^{4}_{y}F(\phi)^{**}_{lp}+O(\Delta y^{4}).
      \end{equation*}
      Substituting these equations into approximation $(\ref{25})$ and rearranging terms to get
      \begin{equation}\label{26}
       \phi^{**}_{lp}=\phi^{*}_{lp}-\frac{k}{2}\delta^{4}_{y}[F(\phi)_{lp}^{**}+F(\phi)_{lp}^{*})]+\frac{k}{2}[G(\phi)_{lp}^{**}+G(\phi)_{lp}^{*}]+O(k^{3}+k\Delta y^{4}).
      \end{equation}
      Truncating the error term $O(k^{3}+k\Delta y^{4})$ and replacing the analytical solution $\phi$ with the computed one $\overline{\phi},$ to obtain
      \begin{equation}\label{27}
       \overline{\phi}^{**}_{lp}=\overline{\phi}^{*}_{lp}-\frac{k}{2}\delta^{4}_{y}[F(\overline{\phi})_{lp}^{**}+F(\overline{\phi})_{lp}^{*})]+
       \frac{k}{2}[G(\overline{\phi})_{lp}^{**}+G(\overline{\phi})_{lp}^{*}].
      \end{equation}
      Equation $(\ref{27})$ yields the implicit formulation satisfied by the difference operator $\overline{\mathcal{P}}_{2}(k)$, that is,
       \begin{equation}\label{28}
       2\overline{\phi}^{**}_{lp}=\overline{\mathcal{P}}_{2}(k)\overline{\phi}^{**}_{lp}+\overline{\mathcal{P}}_{2}(k)\overline{\phi}^{*}_{lp},
      \end{equation}
      for $l=0,1,2,...,M_{x}$, and $p=2,3,...,M_{y}-2$, where $\overline{\mathcal{P}}_{2}(k)$ is defined as
      \begin{equation}\label{29}
       \overline{\mathcal{P}}_{2}(k)\overline{\phi}^{\alpha}_{lp}=\overline{\phi}^{\alpha}_{lp}-\frac{k}{2}\delta^{4}_{y}F(\overline{\phi})_{lp}^{\alpha}+
       \frac{k}{2}G(\overline{\phi})_{lp}^{\alpha},
      \end{equation}
      with $\alpha=*,**$, and
       \begin{equation*}
      F(\overline{\phi})^{\alpha}_{lp}=\begin{bmatrix}
                                 \overline{h}^{\alpha}_{lp}\overline{v}^{\alpha}_{lp} \\
                                 \text{\,}\\
                                 \overline{h}^{\alpha}_{lp}\overline{u}^{\alpha}_{lp}\overline{v}^{\alpha}_{lp} \\
                                 \text{\,}\\
                                 \overline{h}^{\alpha}_{lp}\overline{v}^{2,\alpha}_{lp}+\frac{1}{2}g\overline{h}^{2,\alpha}_{lp} \\
                               \end{bmatrix},\text{\,\,\,\,\,\,}G(\overline{\phi})^{\alpha}_{lp}=\begin{bmatrix}
                                 0 \\
                                 \text{\,}\\
                                 g\overline{h}^{\alpha}_{lp}(\overline{S}_{0_{l}}-\overline{S}_{f_{l}}) \\
                                 \text{\,}\\
                                 g\overline{h}^{\alpha}_{lp}(\overline{S}_{0_{p}}-\overline{S}_{f_{p}})\\
                               \end{bmatrix}.
      \end{equation*}
      Replacing in equations $(\ref{28})$-$(\ref{29})$, $\overline{\phi}^{*}_{lp}$ and $\overline{\phi}^{**}_{lp}$, with $\phi^{*}_{lp}$ and $\phi^{**}_{lp}$, respectively, utilizing equation $(\ref{26})$ and rearranging, this results in
      \begin{equation}\label{30}
       \frac{1}{k}(\phi^{**}_{lp}-\phi^{*}_{lp})=\frac{1}{k}[\mathcal{P}_{2}(k)\phi^{*}_{lp}-\phi^{*}_{lp}]+O(k^{2}+\Delta y^{4}),
      \end{equation}
       where
       \begin{equation}\label{30a}
       \mathcal{P}_{2}(k)\phi^{*}_{lp}=\frac{1}{2}[\overline{\mathcal{P}}_{2}(k)\phi^{**}_{lp}+\overline{\mathcal{P}}_{2}(k)\phi^{*}_{lp}].
      \end{equation}
      Approximation $(\ref{30})$ suggests that the difference scheme defined by equation $(\ref{28})$, when it is stable, should be temporal second-order accurate and spatial fourth-order convergent.\\

       Plugging equations $(\ref{19})$, $(\ref{28})$ and using relation $(\ref{9})$ provides the desired time-split linearized explicit/implicit approach for solving a two-dimensional hydrodynamic flow equation $(\ref{4})$, that is, for $n=0,1,...,N-1$; $l=2,3,...,M_{x}-2$, and $p=2,3,...,M_{y}-2$,
       \begin{equation}\label{31}
       \phi^{*}_{lp}=\mathcal{P}_{1}(k/2)\phi^{n}_{lp},
      \end{equation}
      \begin{equation}\label{32}
       \phi^{**}_{lp}=\mathcal{P}_{2}(k)\phi^{*}_{lp},
      \end{equation}
      \begin{equation}\label{33}
       \phi^{n+1}_{lp}=\mathcal{P}_{1}(k/2)\phi^{**}_{lp},
      \end{equation}
      where the operator $\mathcal{P}_{1}(k/2)$ and $\mathcal{P}_{2}(k)$ are defined by equations $(\ref{20})$ and $(\ref{30a})$, respectively.

      \begin{remark}
      The true initial condition corresponds to a zero depth at any point and should change when the water starts to infiltrate on the ground surface. Since the depth appears in the denominator of many terms, the flow regime is discontinuous and large values of velocity occur. For the sake of numerical simulations we assume that the initial depth is different from zero, but too small. This suggests that small numerical oscillations can destroy the computed solutions. To overcome this drawback, we assume that the source terms are large, can vary more greater in time and space than the other terms in the system of equations $(\ref{1})$. Regarding the boundary conditions, as discussed in \cite{fr}, page 5, we suppose closed boundaries are formed by metal walls on three sides so that there is no flow through the wall (\textbf{Figure 1.i}: Prior to dam-break in Figure $\ref{fig1}$). Thus, the velocities $u$ and $v$ are perpendicular to the boundaries and are set equal zero. Additionally, the first equation in system given by $(\ref{1})$ subjects that a simple integration yields the constant depths at these boundaries. For non closed boundaries (\textbf{Figure 1.ii}: After dam-break in Figure $\ref{fig1}$, the boundary conditions must be well specified.
      \end{remark}

      For this reason, the considered equation $(\ref{4})$ is subjected to the following to initial condition
      \begin{equation}\label{34}
       \overline{\phi}_{lp}^{0}=\rho_{lp},\text{\,\,\,}l=0,1,...,M_{x},\text{\,\,\,}p=0,1,...,M_{y},
      \end{equation}
      and boundary condition
      \begin{equation}\label{35}
      \overline{\phi}_{lp}^{**}=\overline{\phi}_{lp}^{*}=\overline{\phi}_{lp}^{n}=f_{lp}^{n},\text{\,\,\,}l=0,M_{x},\text{\,\,\,}p=0,1,...,M_{y},
      \end{equation}
      \begin{equation}\label{36}
      \overline{\phi}_{lp}^{**}=\overline{\phi}_{lp}^{*}=\overline{\phi}_{lp}^{n}=f_{lp}^{n},\text{\,\,\,}l=0,1,...M_{x},\text{\,\,\,}p=0,M_{y},
      \end{equation}
      where the functions: $\rho=(\rho_{1},\rho_{2},\rho_{3})$ and $f=(f_{1},f_{2},f_{3})$, are given by the initial condition $(\ref{2})$ and boundary condition $(\ref{3b})$. Additionally, to begin the new algorithm, we should set
      \begin{equation}\label{37}
      \overline{\phi}_{lp}^{**}=\overline{\phi}_{lp}^{*}=\overline{\phi}_{lp}^{n}=f_{lp}^{n},\text{\,\,\,\,\,\,}l=1,M_{x}-1,\text{\,\,\,\,\,\,}p=0,1,...,M_{y},
      \end{equation}
      \begin{equation}\label{40}
      \overline{\phi}_{lp}^{**}=\overline{\phi}_{lp}^{*}=\overline{\phi}_{lp}^{n}=f_{lp}^{n},\text{\,\,\,\,\,\,}l=0,1,...,M_{x},\text{\,\,\,\,\,\,}p=1,M_{y}-1.
      \end{equation}

   \section{Overview on the stability analysis}\label{sec3}

   An overland flow is described using the shallow water equations or Saint-Venant system. These equations derive from the three-dimensional incompressible Navier-Stokes equations with some simplifying assumptions such as: the characteristic horizontal size of the field of study is much greater than the water depth, the acceleration due to the pressure balances the gravity, that is, the pressure is hydrostatic and the vertical velocity is negligible and thus has no equation. The considered dam-break flow is a nonlinear hyperbolic equation and the source term makes the overland flow equations more complex. Although a rigorous stability analysis for such equations is exceedingly difficult as already discussed in some previous works (for example, see \cite{1en}), the addition of the source term places extra requirement on the maximum admissible time step for stability. This suggests that the Courant-Friedrichs-Lewy (CFL) condition \cite{26db,41yzw} defined as

   \begin{equation}\label{41}
    \Delta t\leq\min\left\{\frac{\Delta x}{u_{\max}+\sqrt{gh_{\max}}},\frac{\Delta y}{v_{\max}+\sqrt{gh_{\max}}}\right\},
   \end{equation}
    should be considered as a guideline and the maximum admissible (allowable) time step for the developed approach $(\ref{31})$-$(\ref{40})$, has to be less than the predicted CFL condition $(\ref{41})$. Surprisingly, we should establish a time step condition that advances the approximate solution with a maximum time step greater than the one proposed in estimate $(\ref{41})$. The second stage of the new algorithm given by equation $(\ref{32})$ is an implicit finite difference scheme which is known to be unconditionally stable whereas the first and third steps defined by equations $(\ref{31})$ and $(\ref{33})$, respectively, are explicit finite difference formulations and have to require an appropriate time step restriction for their stability. In addition, the proposed approach $(\ref{31})$-$(\ref{40})$ splits the nonlinear system of hyperbolic equations into a series one-dimensional finite difference operators and thereby uses larger time steps for stability. This suggests that the developed technique $(\ref{31})$-$(\ref{40})$, should be more efficient in the computed solution of unsteady flow in the presence of inherent dissipation, discontinuity and stability. To avoid directional bias and to keep the temporal second-order convergence together with the fourth-order accurate in space of the constructed method, a spatial fourth-order difference formula and a temporal second-order approximation are required. Furthermore, being spatial fourth-order accurate, the one-dimensional difference operators $\mathcal{P}_{1}$ and $\mathcal{P}_{2}$, defined by equations $(\ref{20})$ and $(\ref{30a})$, respectively, show that the proposed computational method substantially increases oscillations (oscillation wavelengths equal $12\max\{\Delta x,\Delta y\}$) compared to upwind schemes which can reduce these oscillations. Although the upwind methods are generally first-order accurate, the disturbances in the flow regime that affect convection acceleration cannot be propagated upstream. This shows that they are effective to reduce oscillations.\\

    Now we should state and prove the main result of this paper.

    \begin{theorem} \label{t1} (time step restriction for stability).
    Suppose that $\overline{\phi}=(\overline{h},\overline{h}\overline{u},\overline{h}\overline{v})^{t}$, is the approximate solution provided by the developed time-split linearized explicit/implicit approach $(\ref{31})$-$(\ref{40})$, for solving the two-dimensional hydrodynamic problem $(\ref{1})$, with initial-boundary conditions $(\ref{2})$-$(\ref{3b})$. Thus, the new algorithm $(\ref{31})$-$(\ref{40})$ is stable under the following time step restriction
    \begin{equation}\label{43}
    k\leq \frac{48}{\gamma}\min\left\{\frac{\|\overline{\beta}\|_{0}}{\sqrt{M_{x}-3}\||\overline{u}|\|_{0,\infty}},
    \frac{\||\overline{u}|\|_{0,\infty}}{\||\overline{u}^{2}+\frac{1}{2}g\overline{h}|\|_{0,\infty}}\right\}\Delta x,
   \end{equation}
    \end{theorem}
    where $\overline{\beta}=1\in L^{2}(\Omega)$, $0<\gamma\leq 18$, $g$ is the acceleration of gravity, $k$ denotes the time step and $\Delta x$ represents the mesh grid in the $x$-direction, $M_{x}>3$ is an integer.

   \begin{remark}
   One should observe that if $(M_{y}-3)\Delta y\leq b$, where $b$ is a positive parameter (which is true whenever $\Omega=(a_{1},a_{2})\times(b_{1},b_{2})$, and the mesh size $\Delta y$ in the $y$-direction is defined as $\Delta y=\frac{b_{2}-b_{1}}{M_{y}}$ with $b=b_{2}-b_{1}$), then $\frac{\|\overline{\beta}\|_{0}}{\sqrt{M_{x}-3}}\leq \sqrt{b\Delta x}$.
   \end{remark}

    The following Definition and Lemmas are very important in the proof of Theorem $\ref{t1}$.

    \begin{definition}\label{d1} \cite{sb}
    Let $A$ be an $n\times n$ matrix, then the matrix $A$ is normal if
    \begin{equation*}
     A^{H}A=AA^{H},
   \end{equation*}
   where $A^{H}$ denotes the transpose conjugate of $A$.
   \end{definition}

   \begin{lemma}\label{l2} (Ger\v{s}gorin theorem)\cite{sb}.
    Consider $A=[a_{ij}]$ be an $n\times n$ matrix. Let $D_{i}$ be the disc in the complex field $\mathbb{C}$ centered at $a_{ii}$ with radius $\underset{\underset{j\neq i}{j=1}}{\overset{n}\sum}|a_{ij}|$, that is,
   \begin{equation*}
    D_{i}=\left\{\lambda\in\mathbb{C}:\text{\,}|\lambda-a_{ii}|\leq\underset{\underset{j\neq i}{j=1}}{\overset{n}\sum}|a_{ij}|\right\}.
   \end{equation*}
    Thus, all the eigenvalues of $A$ are contained in the union of discs $D_{i}$, for $i=1,2,...,n$.
   \end{lemma}

     \begin{lemma}\label{l3} \cite{sb}
    Every Hermitian matrix $C$ of size $n\times n$ is diagonalizable and its maximum and minimum eigenvalues $\rho_{\max}$ and $\sigma_{\min}$, respectively, are given by
   \begin{equation*}
    \rho_{\max}=\underset{0\neq z\in \mathbb{C}^{n}}{\max}\frac{z^{H}Cz}{z^{H}z},\text{\,\,\,\,and\,\,\,\,}\sigma_{\min}=\underset{0\neq z\in \mathbb{C}^{n}}{\min}\frac{z^{H}Cz}{z^{H}z}.
   \end{equation*}
   Furthermore, any square matrix $C$ that is normal is also diagonalizable.
   \end{lemma}

    \begin{lemma}\label{l4}
    Suppose that $C$ is an $n\times n$ pentadiagonal matrix defined as

   \begin{equation*}
    C=\begin{bmatrix}
        \alpha_{0} & \alpha_{1} & -\alpha_{2} & 0 & \cdots & 0 \\
        -\alpha_{1} & \alpha_{0} & \alpha_{1} & -\alpha_{2} & \ddots & \vdots \\
        \alpha_{2} & -\alpha_{1} & \alpha_{0} & \alpha_{1} & \ddots & 0 \\
        0 & \alpha_{0} & \ddots & \ddots & \ddots & -\alpha_{2} \\
        \vdots & \ddots & \ddots & \ddots & \ddots & \alpha_{1} \\
        0 & \ldots & 0 & \alpha_{2} & -\alpha_{1} & \alpha_{0} \\
      \end{bmatrix}.
   \end{equation*}
    Thus, $C$ is normal and its maximum eigenvalue satisfies
    \begin{equation*}
     0<\||C|\|=\rho_{\max}(C)\leq \alpha_{0}+2(|\alpha_{1}|+|\alpha_{2}|),
    \end{equation*}
    where $\||\cdot|\|$ is the matrix norm defined by equation $(\ref{12a})$.
   \end{lemma}

   \begin{proof}
    Firstly, it's easy to see that: $C^{t}=-C$, so $C^{t}C=-C^{2}=CC^{t}$. Hence, the matrix $C$ is normal. Furthermore, it is not hard to observe that $C^{t}C$ is symmetric. Thus, $C^{t}C$ is diagonalizable and its maximum eigenvalue $\rho_{\max}(C^{t}C)$ is given by
    \begin{equation}\label{54}
    \rho_{\max}(C^{t}C)=\underset{0\neq z\in \mathbb{R}^{n}}{\max}\frac{z^{t}C^{t}Cz}{z^{t}z}.
   \end{equation}

   Since $z^{t}z=\|z\|_{\sim}^{2}$ and $z^{t}C^{t}Cz=\|Cz\|_{\sim}^{2}$, where $\|\cdot\|_{\sim}$ is the vector norm defined by equation $(\ref{11a})$. This fact combined with equations $(\ref{12a})$ and $(\ref{54})$, provide $\rho_{\max}(C^{t}C)=\||C^{t}C|\|=\||-C|\|^{2}=\||C|\|^{2}$. This is equivalent to
   \begin{equation}\label{55}
    \||C|\|=\sqrt{\rho_{\max}(C^{t}C)}.
   \end{equation}

   Because $C$ is normal, it follows from Lemma $\ref{l3}$ that the matrix $C$ is diagonalizable. So, there is an orthogonal matrix $B$ of size $n\times n$, so that
   \begin{equation*}
     C=Bdiag(\lambda_{1},...,\lambda_{n})B^{t},
    \end{equation*}
    where $diag(\lambda_{1},...,\lambda_{n})$ is the diagonal matrix whose diagonal elements are $\lambda_{j}$, $j=1,2,...,n$, which are the eigenvalues of $C$. But $B^{t}B=I_{n}=BB^{t}$, where $I_{n}$ is the identity matrix of size $n\times n$. So,
    \begin{equation*}
     C^{t}C=Bdiag(\lambda_{1},...,\lambda_{n})B^{t}Bdiag(\lambda_{1},...,\lambda_{n})B^{t}=Bdiag(\lambda_{1}^{2},...,\lambda_{n}^{2})B^{t}.
    \end{equation*}

    Thus, $\rho_{\max}(C^{t}C)=\rho_{\max}(C)^{2}$. But it follows from the Ger\v{s}gorin result given by Lemma $\ref{l2}$ that all the eigenvalues of $C$ are contained in the discs centered at $\alpha_{0}$ with radius $2(|\alpha_{1}|+|\alpha_{2}|)$. Hence, $|\rho_{\max}(C)-\alpha_{0}|\leq 2(|\alpha_{1}|+|\alpha_{2}|)$, which is equivalent to: $0<\rho_{\max}(C)\leq \alpha_{0}+2(|\alpha_{1}|+|\alpha_{2}|)$. This fact, together with equation $(\ref{55})$ complete the proof of Lemma $\ref{l4}$.
   \end{proof}

    \begin{proof} (of Theorem $\ref{t1}$).
      We recall that the second stage $(\ref{32})$ of the constructed technique $(\ref{31})$-$(\ref{33})$ is an implicit difference formulation, so the difference scheme $(\ref{32})$ is unconditionally stable. Additionally, the first step $(\ref{31})$ and third one $(\ref{33})$ are explicit difference schemes, provide that they should be stable under a time step requirement. Since equations $(\ref{31})$ and $(\ref{33})$ are defined by the same operator $\mathcal{P}_{1}(k/2)$ given by equation $(\ref{20})$, then both first and third stages of the numerical method are stable if
      \begin{equation}\label{44a}
      \||\mathcal{P}_{1}(k/2)\phi_{p}^{\theta}|\|_{0,\sim}<||\phi_{p}^{\theta}|\|_{0,\sim},
      \end{equation}
      for $p=2,3,...,M_{y}-2$, and $\theta\in\{n,**\}$, where the norm $\||\cdot|\|_{0,\sim}$, is defined in equation $(\ref{10a})$. We have to find a time step requirement for which estimate $(\ref{44a})$ holds only in the case $\theta=n$. The case $\theta=**$ gives the same result.\\

      For $\theta=n$, plugging equation $(\ref{20})$ and estimate $(\ref{44a})$, this results in
      \begin{equation}\label{44}
      \||\phi^{n}_{p}-\frac{k}{2}\delta^{4}_{x}E(\phi^{n}_{p})+\frac{k^{2}}{8}\delta^{2}_{x}[\nabla_{\phi}E(\phi^{n}_{p})\cdot\delta^{3\mp}_{x}E(\phi^{n}_{p})]|\|_{0,\sim}
      <||\phi_{p}^{n}|\|_{0,\sim}.
      \end{equation}

      Set
      \begin{equation*}
      \overline{h}_{p}^{n}=[\overline{h}_{2p}^{n},...,\overline{h}_{M_{x}-2,p}^{n}]^{t},\text{\,\,}
      \overline{u}_{p}^{n}=[\overline{u}_{2p}^{n},...,\overline{u}_{M_{x}-2,p}^{n}]^{t},\text{\,\,}
      \overline{v}_{p}^{n}=[\overline{v}_{2p}^{n},...,\overline{v}_{M_{x}-2,p}^{n}]^{t},\text{\,\,}
      E_{1p}^{n}=[\overline{h}_{2p}^{n}\overline{u}_{2p}^{n},...,\overline{h}_{M_{x}-2,p}^{n}\overline{u}_{M_{x}-2,p}^{n}],
    \end{equation*}
    \begin{equation}\label{45}
      E_{2p}^{n}=[\overline{h}_{2p}^{n}\overline{u}_{2p}^{2,n}+\frac{1}{2}g\overline{h}_{2p}^{2,n},...,
      \overline{h}_{M_{x}-2,p}^{n}\overline{u}_{M_{x}-2,p}^{2,n}+\frac{1}{2}g\overline{h}_{M_{x}-2,p}^{2,n}]^{t},\text{\,\,\,}
      E_{3p}^{n}=[\overline{h}_{2p}^{n}\overline{u}_{2p}^{n}\overline{v}_{2p}^{n},...,\overline{h}_{M_{x}-2,p}^{n}\overline{u}_{M_{x}-2,p}^{n}\overline{v}_{M_{x}-2,p}^{n}]^{t}.
      \end{equation}

      Utilizing equation $(\ref{21})$, simple calculations yield
      \begin{equation*}
      \nabla_{\phi}E(\overline{\phi})^{n}_{lp}=\begin{bmatrix}
      \overline{u}^{n}_{lp} & \overline{h}^{n}_{lp} & 0 \\
       \text{\,}\\
       \overline{u}^{2,n}_{lp}+g\overline{h}^{n}_{lp} & 2\overline{h}^{n}_{lp}\overline{u}^{n}_{lp} & 0 \\
       \text{\,}\\
       \overline{u}^{n}_{lp}\overline{v}^{n}_{lp} & \overline{h}^{n}_{lp}\overline{v}^{n}_{lp} & \overline{h}^{n}_{lp}\overline{u}^{n}_{lp} \\
       \end{bmatrix}\begin{bmatrix}
                      \delta^{3\mp}_{x}E_{1lp}^{n} \\
                      \text{\,}\\
                      \delta^{3\mp}_{x}E_{2lp}^{n} \\
                      \text{\,}\\
                      \delta^{3\mp}_{x}E_{3lp}^{n} \\
                    \end{bmatrix}=\begin{bmatrix}
                      \overline{u}^{n}_{lp}\delta^{3\mp}_{x}E_{1lp}^{n}+\overline{h}^{n}_{lp}\delta^{3\mp}_{x}E_{1lp}^{n} \\
                      \text{\,}\\
                      (\overline{u}^{2,n}_{lp}+g\overline{h}^{n}_{lp})\delta^{3\mp}_{x}E_{1lp}^{n}+2\overline{h}^{n}_{lp}\overline{u}^{n}_{lp}\delta^{3\mp}_{x}E_{2lp}^{n} \\
                      \text{\,}\\
                      \overline{u}^{n}_{lp}\overline{v}^{n}_{lp}\delta^{3\mp}_{x}E_{1lp}^{n}+\overline{h}^{n}_{lp}\overline{v}^{n}_{lp}\delta^{3\mp}_{x}E_{2lp}^{n}+
                      \overline{h}^{n}_{lp}\overline{u}^{n}_{lp}E_{3lp}^{n} \\
                    \end{bmatrix},
      \end{equation*}
      for $l=2,3,...,M_{x}-2$, and $p=2,3,...,M_{y}-2$. Substituting this into estimate $(\ref{44})$ and using the definition of the $[\mathbb{R}^{M_{x}-3}]^{3}$-norm, $\||\cdot|\|_{0,\sim}$, defined in equation $(\ref{10a})$ to get
      \begin{equation*}
      \|\overline{h}^{n}_{p}-\frac{k}{2}[\delta^{4}_{x}E_{1p}^{n}-\frac{k}{4}\delta^{2}_{x}(\overline{u}^{n}_{p}\delta^{3\mp}_{x}E^{n}_{1p}+\overline{h}^{n}_{p}
      \delta^{3\mp}_{x}E^{n}_{2p})]\|_{\sim}^{2}+\|(\overline{h}\overline{u})^{n}_{p}-\frac{k}{2}[\delta^{4}_{x}E_{2p}^{n}-\frac{k}{4}\delta^{2}_{x}
      ((\overline{u}^{2}+g\overline{h})_{p}^{n}\delta^{3\mp}_{x}E^{n}_{1p}+2(\overline{h}\overline{v})^{n}_{p}\delta^{3\mp}_{x}E^{n}_{2p})]\|_{\sim}^{2}+
    \end{equation*}
      \begin{equation}\label{45a}
      \|(\overline{h}\overline{v})^{n}_{p}-\frac{k}{2}[\delta^{4}_{x}E_{3p}^{n}-\frac{k}{4}\delta^{2}_{x}((\overline{u}\overline{v})_{p}^{n}\delta^{3\mp}_{x}E^{n}_{1p}+
      (\overline{h}\overline{v})_{p}^{n}\delta^{3\mp}_{x}E^{n}_{2p}+(\overline{h}\overline{u})^{n}_{p}\delta^{3\mp}_{x}E^{n}_{3p})]\|_{\sim}^{2}<
      \|\overline{h}^{n}_{p}\|_{\sim}^{2}+\|(\overline{h}\overline{u})^{n}_{p}\|_{\sim}^{2}+\|(\overline{h}\overline{v})^{n}_{p}\|_{\sim}^{2}.
      \end{equation}

      Because we are interested in the allowable time step requirement for which estimate $(\ref{45a})$ is satisfied, we have to find a maximum time step restriction that satisfies the following three inequalities:
       \begin{equation}\label{46}
      \|\overline{h}^{n}_{p}-\frac{k}{2}[\delta^{4}_{x}E_{1p}^{n}-\frac{k}{4}\delta^{2}_{x}(\overline{u}^{n}_{p}\delta^{3\mp}_{x}E^{n}_{1p}+\overline{h}^{n}_{p}
      \delta^{3\mp}_{x}E^{n}_{2p})]\|_{\sim}^{2}<\|\overline{h}^{n}_{p}\|_{\sim}^{2},
      \end{equation}
       \begin{equation}\label{47}
      \|(\overline{h}\overline{u})^{n}_{p}-\frac{k}{2}[\delta^{4}_{x}E_{2p}^{n}-\frac{k}{4}\delta^{2}_{x}((\overline{u}^{2}+g\overline{h})_{p}^{n}\delta^{3\mp}_{x}E^{n}_{1p}+
      2(\overline{h}\overline{v})^{n}_{p}\delta^{3\mp}_{x}E^{n}_{2p})]\|_{\sim}^{2}<\|(\overline{h}\overline{u})^{n}_{p}\|_{\sim}^{2},
      \end{equation}
       \begin{equation}\label{48}
      \|(\overline{h}\overline{v})^{n}_{p}-\frac{k}{2}[\delta^{4}_{x}E_{3p}^{n}-\frac{k}{4}\delta^{2}_{x}((\overline{u}\overline{v})_{p}^{n}\delta^{3\mp}_{x}E^{n}_{1p}+
      (\overline{h}\overline{v})_{p}^{n}\delta^{3\mp}_{x}E^{n}_{2p}+(\overline{h}\overline{u})^{n}_{p}\delta^{3\mp}_{x}E^{n}_{3p})]\|_{\sim}^{2}<
      \|(\overline{h}\overline{v})^{n}_{p}\|_{\sim}^{2}.
      \end{equation}

      Indeed, summing inequalities $(\ref{46})$-$(\ref{48})$ side by side gives estimate $(\ref{45a})$.\\

      Expanding the left side of estimate $(\ref{46})$ and after simplification, we obtain
       \begin{equation*}
      \frac{k^{2}}{4}\|\delta^{4}_{x}E_{1p}^{n}-\frac{k}{4}\delta^{2}_{x}(\overline{u}^{n}_{p}\delta^{3\mp}_{x}E^{n}_{1p}+\overline{h}^{n}_{p}\delta^{3\mp}_{x}E^{n}_{2p})\|_{\sim}^{2}
      -k\left(\overline{h}^{n}_{p},\delta^{4}_{x}E_{1p}^{n}-\frac{k}{4}\delta^{2}_{x}(\overline{u}^{n}_{p}\delta^{3\mp}_{x}E^{n}_{1p}+\overline{h}^{n}_{p}
      \delta^{3\mp}_{x}E^{n}_{2p})\right)_{\sim}<0,
      \end{equation*}
      which is equivalent to
        \begin{equation*}
      k\|\delta^{4}_{x}E_{1p}^{n}-\frac{k}{4}\delta^{2}_{x}(\overline{u}^{n}_{p}\delta^{3\mp}_{x}E^{n}_{1p}+\overline{h}^{n}_{p}\delta^{3\mp}_{x}E^{n}_{2p})\|_{\sim}^{2}
      <4\left(\overline{h}^{n}_{p},\delta^{4}_{x}E_{1p}^{n}-\frac{k}{4}\delta^{2}_{x}(\overline{u}^{n}_{p}\delta^{3\mp}_{x}E^{n}_{1p}+\overline{h}^{n}_{p}
      \delta^{3\mp}_{x}E^{n}_{2p})\right)_{\sim}.
      \end{equation*}

      Applying the Cauchy-Schwarz inequality and simplifying, this implies
      \begin{equation}\label{49}
      k\|\delta^{4}_{x}E_{1p}^{n}-\frac{k}{4}\delta^{2}_{x}(\overline{u}^{n}_{p}\delta^{3\mp}_{x}E^{n}_{1p}+\overline{h}^{n}_{p}\delta^{3\mp}_{x}E^{n}_{2p})\|_{\sim}
      <4\|\overline{h}^{n}_{p}\|_{\sim}.
      \end{equation}

      Of course the aim of this paper is to give a general picture of necessary condition of stability. Since the formulae can become quite heavy, for the convenient of writing and for small values of the time step $k$, it holds
      \begin{equation}\label{50}
      \delta^{4}_{x}E_{1p}^{n}\approx \delta^{4}_{x}E_{1p}^{n}-\frac{k}{4}\delta^{2}_{x}(\overline{u}^{n}_{p}\delta^{3\mp}_{x}E^{n}_{1p}
      +\overline{h}^{n}_{p}\delta^{3\mp}_{x}E^{n}_{2p}).
      \end{equation}

      However, the truncation of the infinitesimal term $\frac{k}{4}\delta^{2}_{x}(\overline{u}^{n}_{p}\delta^{3\mp}_{x}E^{n}_{1p}+\overline{h}^{n}_{p}\delta^{3\mp}_{x}
      E^{n}_{2p})$, does not compromise the result on stability. Utilizing approximation $(\ref{50})$, estimate $(\ref{49})$ becomes
       \begin{equation}\label{51}
      k\|\delta^{4}_{x}E_{1p}^{n}\|_{\sim} \leq 4\|\overline{h}^{n}_{p}\|_{\sim}.
      \end{equation}

      It's not hard to observe that using the linear operator $\delta^{4}_{x}$ defined in relation $(\ref{9a})$, the term $\delta^{4}_{x}E_{1p}^{n}$ can be expressed in the matrix form as
      \begin{equation}\label{52}
      \delta^{4}_{x}E_{1p}^{n}=\frac{1}{12\Delta x}\underset{A}{\underbrace{\begin{bmatrix}
        0 & 8 & -1 & 0 & \cdots & 0 \\
        -8 & 0 & 8 & -1 & \ddots & \vdots \\
        1 & -8 & 0 & 8 & \ddots & 0 \\
        0 & 1 & \ddots & \ddots & \ddots & -1 \\
        \vdots & \ddots & \ddots & \ddots & \ddots & 8 \\
        0 & \ldots & 0 & 1 & -8 & 0 \\
       \end{bmatrix}}}\underset{E_{1p}^{n}}{\underbrace{\begin{bmatrix}
                     (\overline{h}\overline{u})^{n}_{2p} \\
                     \text{\,}\\
                     (\overline{h}\overline{u})^{n}_{3p} \\
                     \text{\,}\\
                     \vdots \\
                     \vdots \\
                     (\overline{h}\overline{u})^{n}_{M_{x}-2,p} \\
                   \end{bmatrix}}}=\frac{1}{12\Delta x}AE_{1p}^{n}.
     \end{equation}

     So,
      \begin{equation}\label{53}
      \|\delta^{4}_{x}E_{1p}^{n}\|_{\sim}\leq \frac{1}{12\Delta x}\|AE_{1p}^{n}\|_{\sim}\leq \frac{1}{12\Delta x}\||A|\|\|E_{1p}^{n}\|_{\sim}=
      \frac{1}{12\Delta x}\||A|\|\|\overline{h}^{n}_{p}\overline{u}^{n}_{p}\|_{\sim}\leq \frac{1}{12\Delta x}\||A|\|\|\overline{h}^{n}_{p}\|_{\sim}
      \|\overline{u}^{n}_{p}\|_{\sim},
      \end{equation}
      where $\||\cdot|\|$ denotes the matrix norm defined by equation $(\ref{12a})$ and associated with the norm $\|\cdot\|_{\sim}$. Since $A^{t}=-A$, then $A^{t}A=AA^{t}$. So, $A$ is a pentadiagonal matrix which is normal. It follows from Lemma $\ref{l4}$ that $0<\||A|\|=\rho_{\max}(A)\leq 2(1+8)=18$. This fact, together with estimate $(\ref{53})$ result in
      \begin{equation*}
      \|\delta^{4}_{x}E_{1p}^{n}\|_{\sim}\leq \frac{1}{12\Delta x}\rho_{\max}(A)\|\overline{h}^{n}_{p}\|_{\sim}\|\overline{u}^{n}_{p}\|_{\sim}.
      \end{equation*}

      Utilizing this, a condition for which inequality $(\ref{51})$ holds are the values of $k$ that satisfy
      \begin{equation*}
      \frac{k}{12\Delta x}\rho_{\max}(A)\|\overline{h}^{n}_{p}\|_{\sim}\|\overline{u}^{n}_{p}\|_{\sim}\leq 4\|\overline{h}^{n}_{p}\|_{\sim}.
      \end{equation*}

      Multiplying both sides of this estimate by $\|\overline{h}^{n}_{p}\|_{\sim}^{-1}$ and using the definition of the norm, $\|\cdot\|_{\sim}$, we obtain
      \begin{equation*}
      \frac{k}{12\Delta x}\rho_{\max}(A)\left(\underset{l=2}{\overset{M_{x}-2}\sum}(\overline{u}^{n}_{lp})^{2}\right)^{\frac{1}{2}}\leq 4.
      \end{equation*}

       Squared both sides of this estimate yields
       \begin{equation}\label{57}
      \frac{k^{2}}{(12\Delta x)^{2}}\rho_{\max}(A)^{2}\underset{l=2}{\overset{M_{x}-2}\sum}(\overline{u}^{n}_{lp})^{2}\leq 16.
      \end{equation}

      Summing up inequality $(\ref{57})$, for $p=2,3,...,M_{y}-2$, multiplying both sides of the new estimate by $\Delta x\Delta y$ and utilizing the norm, $\|\cdot\|_{0}$, defined in relation $(\ref{10})$, this gives
       \begin{equation*}
      \frac{k^{2}}{(12\Delta x)^{2}}\rho_{\max}(A)^{2}\|\overline{u}^{n}\|_{0}^{2}\leq 16\Delta x\Delta y\underset{p=2}{\overset{M_{y}-2}\sum}1\leq \frac{16}{M_{x}-3}\Delta x\Delta y\underset{p=2}{\overset{M_{y}-2}\sum}\underset{l=2}{\overset{M_{x}-2}\sum}1=\frac{16}{M_{x}-3}\|\overline{\beta}\|_{0}^{2},
      \end{equation*}
      where $\overline{\beta}=1\in L^{2}(\Omega)$. The square root of this estimate provides
       \begin{equation}\label{58}
      \frac{k}{12\Delta x}\rho_{\max}(A)\|\overline{u}^{n}\|_{0}\leq \frac{4}{\sqrt{M_{x}-3}}\|\overline{\beta}\|_{0}.
      \end{equation}

      Taking the maximum over $n$, for $n=0,1,...,N$, using the norm $\||\cdot|\|_{0,\infty}$, defined by equation $(\ref{11})$, and rearranging terms, inequality $(\ref{58})$ results in
      \begin{equation}\label{59}
       k\leq \frac{48\|\overline{\beta}\|_{0}\Delta x}{\rho_{\max}(A)\||\overline{u}|\|_{0,\infty}\sqrt{M_{x}-3}},
      \end{equation}
      where $0<\rho_{\max}(A)\leq18$, and $\overline{\beta}=1\in L^{2}(\Omega)$.\\

      Now, utilizing estimate $(\ref{47})$, it holds
      \begin{equation*}
      \|(\overline{h}\overline{u})^{n}_{p}-\frac{k}{2}[\delta^{4}_{x}E_{2p}^{n}-\frac{k}{4}\delta^{2}_{x}((\overline{u}^{2}+g\overline{h})_{p}^{n}\delta^{3\mp}_{x}E^{n}_{1p}+
      2(\overline{h}\overline{v})^{n}_{p}\delta^{3\mp}_{x}E^{n}_{2p})]\|_{\sim}^{2}<\|(\overline{h}\overline{u})^{n}_{p}\|_{\sim}^{2},
      \end{equation*}

        Expanding this estimate, applying the Cauchy-Schwarz inequality, simplifying and rearranging terms to obtain
       \begin{equation}\label{60}
      k\|\delta^{4}_{x}E_{2p}^{n}-\frac{k}{4}\delta^{2}_{x}((\overline{u}^{2}+g\overline{h})_{p}^{n}\delta^{3\mp}_{x}E^{n}_{1p}
      +2(\overline{h}\overline{v})^{n}_{p}\delta^{3\mp}_{x}E^{n}_{2p})]\|_{\sim}<4\|(\overline{h}\overline{u})^{n}_{p}\|_{\sim}\leq 4\|\overline{h}^{n}_{p}\|_{\sim}
      \|\overline{u}^{n}_{p}\|_{\sim}.
      \end{equation}

      For the values of time step $k$ small enough, the following approximation is satisfied
      \begin{equation}\label{61}
       \delta^{4}_{x}E_{2p}^{n}-\frac{k}{4}\delta^{2}_{x}((\overline{u}^{2}+g\overline{h})_{p}^{n}\delta^{3\mp}_{x}E^{n}_{1p}
      +2(\overline{h}\overline{v})^{n}_{p}\delta^{3\mp}_{x}E^{n}_{2p}) \approx \delta^{4}_{x}E_{2p}^{n}.
      \end{equation}

      For the reason mentioned above, the truncation of the infinitesimal term: $\frac{k}{4}\delta^{2}_{x}((\overline{u}^{2}+g\overline{h})_{p}^{n}\delta^{3\mp}_{x}E^{n}_{1p}
      +2(\overline{h}\overline{v})^{n}_{p}\delta^{3\mp}_{x}E^{n}_{2p})$, does not compromise the result on stability. Using approximation $(\ref{61})$, estimate $(\ref{60})$ becomes
       \begin{equation}\label{62}
      k\|\delta^{4}_{x}E_{2p}^{n}\|_{\sim} \leq 4\|\overline{h}^{n}_{p}\|_{\sim}\|\overline{u}^{n}_{p}\|_{\sim}.
      \end{equation}

      It follows from equation $(\ref{45})$ that $E_{2p}^{n}=(\overline{u}_{p}^{2,n}+\frac{1}{2}g\overline{h}_{p}^{n})\overline{h}^{n}_{p}$. Utilizing equation $(\ref{52})$ and replacing the term $E_{1p}^{n}$ with $E_{2p}^{n}$, straightforward calculations result in

     So,
      \begin{equation*}
      \|\delta^{4}_{x}E_{2p}^{n}\|_{\sim}\leq \frac{1}{12\Delta x}\|AE_{2p}^{n}\|_{\sim}\leq \frac{1}{12\Delta x}\||A|\|\|(\overline{u}_{p}^{2,n}+\frac{1}{2}g\overline{h}_{p}^{n})\overline{h}^{n}_{p}\|_{\sim}\leq \frac{\rho_{\max}(A)}{12\Delta x}\|\overline{h}^{n}_{p}\|_{\sim}\|\overline{u}_{p}^{2,n}+\frac{1}{2}g\overline{h}_{p}^{n}\|_{\sim}.
      \end{equation*}

      Since we are interested in a necessary condition on the time step $k$ for stability, the values of $k$ which satisfy the following estimate will ensure inequality $(\ref{62})$.
      \begin{equation*}
      \frac{k}{12\Delta x}\rho_{\max}(A)\|\overline{h}^{n}_{p}\|_{\sim}\|\overline{u}_{p}^{2,n}+\frac{1}{2}g\overline{h}_{p}^{n}\|_{\sim}\leq
      4\|\overline{h}^{n}_{p}\|_{\sim}\|\overline{u}^{n}_{p}\|_{\sim}.
      \end{equation*}

      This is equivalent to
      \begin{equation*}
      \frac{k}{12\Delta x}\rho_{\max}(A)\|\overline{u}_{p}^{2,n}+\frac{1}{2}g\overline{h}_{p}^{n}\|_{\sim}\leq 4\|\overline{u}^{n}_{p}\|_{\sim}.
      \end{equation*}

       Taking the square in both sides of this estimate, multiplying the obtained inequality by $\Delta x\Delta y$, summing up, from $p=2,3,...,M_{y}-2$, and utilizing the norm, $\|\cdot\|_{0}$, defined in equation $(\ref{10})$, to get
       \begin{equation*}
      \frac{k^{2}}{(12\Delta x)^{2}}\rho_{\max}(A)^{2}\|\overline{u}^{2,n}+\frac{1}{2}g\overline{h}^{n}\|_{0}^{2}\leq 16 \|\overline{u}^{n}\|_{0}^{2},
      \end{equation*}

       The square root in both sides of this inequality yields
       \begin{equation*}
      \frac{k}{12\Delta x}\rho_{\max}(A)\|\overline{u}^{2,n}+\frac{1}{2}g\overline{h}^{n}\|_{0}\leq 4\|\overline{u}^{n}\|_{0}.
      \end{equation*}

      Taking the maximum over $n$, for $0\leq n\leq N$, rearranging terms, and utilizing the norm, $\||\cdot|\|_{0,\infty}$, defined by equation $(\ref{11})$, result in
      \begin{equation}\label{63}
       k\leq \frac{48\Delta x}{\rho_{\max}(A)}\frac{\|\overline{u}\|_{0,\infty}}{\|\overline{u}^{2,n}+\frac{1}{2}g\overline{h}\|_{0,\infty}},
      \end{equation}
      where $0<\rho_{\max}(A)\leq18$.\\

      In a similar manner, one easily shows that a necessary condition for which inequality $(\ref{48})$ holds is defined as
      \begin{equation}\label{64}
       k\leq \frac{48\|\overline{\beta}\|_{0}\Delta x}{\rho_{\max}(A)\||\overline{u}|\|_{0,\infty}\sqrt{M_{x}-3}},
      \end{equation}
      where $0<\rho_{\max}(A)\leq18$, and $\overline{\beta}=1\in L^{2}(\Omega)$. A combination of estimates $(\ref{59})$ and $(\ref{63})$-$(\ref{64})$, completes the proof of Theorem $\ref{t1}$.
      \end{proof}

      \section{Numerical experiments}\label{sec4}

      This section simulates a time-split linearized explicit/implicit numerical method $(\ref{31})$-$(\ref{40})$ applied to two-dimensional hydrodynamic flow equations $(\ref{1})$ subjects to initial-boundary conditions $(\ref{2})$-$(\ref{3b})$. Two numerical examples described in \cite{48tdc} are carried out to confirm the theoretical studies and to demonstrate the utility and efficiency of the new algorithm $(\ref{31})$-$(\ref{40})$. Furthermore, the developed computational approach is used to investigate and forecast the practical case of floods observed in Cameroon far north region from the second half of July up to the second half of October $2024$. Due to heavy rains, series of floods are recorded in Logone-et-Chari and Mayo-Danay divisions. The Logone sources are delineated in the northern Cameroon, western Central African Republic and southern Chad. The Logone river or Logon considers two main tributaries so called: the Mbere river also known as western Logone which is located in the east Cameroon and the Pende river or eastern Logone in the prefecture Ouham-Pende and located in the Central African Republic \cite{2lc} while the Chari river is a $1400km$ long river flowing in Central African. More precisely, it flows from the Central African Republic through Chad into lake Chad, the Cameroon border from N'djamena and is joined by the Logone river, its western and principal tributary (see Figure \ref{fig2}, \textbf{Figure 2.v}). From $1951$ till $1984$, it has been observed a flow of the river in a town (Bongor) which is in Chad downstream of the union with the Pende about $450km$ above the mouth into the Chari. The average annual flow observed during this period is $492m^{3}/s$ fed by an area of about $73.7km^{2}$, approximately $94.5\%$ of the total catchment area of the river. The river overflowed also affected the communities in Mayo-Danay department. From 17-18 September 2013, floods caused a rupture of the dam along the Logon in the Dougui town and Kai Kai district in the far north region of Cameroon. Additionally, a second rupture in the dam $4km$ from the first rupture has started flooding the area on $27$ September $2013$, and approximately $9000$ persons were displaced \cite{4lc}. In the numerical simulations of floods from July to October $2024$, the initial and boundary conditions are obtained from the minimum, average and maximum of the annual discharges provided by GRDC station information in $2013$ \cite{3lc} together with other data. In addition, the study considers Logone-et-Chari and Mayo-Danay which deal with one river and one lake, so called Logone river (see Figure \ref{fig2}, \textbf{Figure 2.v}) and Guere lake (see Figure \ref{fig2}, \textbf{Figure 2.vi}). Specifically, the location of the catchment is the Logone basin. The GRDC catchment area is approximately $73700km^{2}$ with a GRDC interstation area of $25430km^{2}$ while the GRDC station runoff and GRDC interstation runof are $210mm/year$ and $-3mm/year$, respectively. The minimum and maximum discharges are $16m^{3}/s$ and $2420m^{3}/s$, respectively, whereas the mean discharge and the mean interstation one equal $492m^{3}/s$ and $-2m^{3}/s$, respectively. Furthermore, the distance to the next downstream station is $445km$ while the distance to the basin outlet equals $579km$ and the length of upstream main stem is $511km$. The water depth at $t=0$ is specified as the initial condition in the downstream. The velocities $u$ and $v$ at the inflow boundary are considered to be zero at any time while the initial velocities are obtained from the following three cases: minimum annual discharge, average annual discharge and maximum annual discharge observed in $2013$. It's worth mentioning that the three considered cases of initial condition (minimum, average and maximum annual discharges observed in $2013$) allow to assess and predict both water depth and velocities during different seasons and thus, investigating the minimum and maximum flood extents which represent the main tools to forecast several potential inundated locations in Cameroon far north region.\\

      \subsection*{Thacker's analytical solutions}
      Thacker developed a wide set of analytical solutions for two-dimensional shallow water models with moving boundaries including the curiolis force \cite{48tdc}. Since these solutions are not discontinuous and deal with both bed slope and wetting/drying with two dimensional effects, they are suitable to verify the efficiency and accuracy of numerical schemes. To verify the the stability and convergence order of the proposed technique under the time step requirement given by estimate $(\ref{43})$ in Theorem $\ref{t1}$. We assume that the space steps $\Delta x$ and $\Delta y$ are equal and vary in the range: $3^{-l},$ for $l=2,\cdots,6,$ while the time step $\Delta t\in\{3^{-l},\text{\,\,}l=4,\cdots,8\}$. We compute the error: $e_{(\cdot)}=\overline{w}-w$, where $w\in\{h,u,v\}$ and $\overline{w}\in\{\overline{h},\overline{u},
      \overline{v}\}$, utilizing the $L^{\infty}$-norm, $\|\cdot\|_{0,\infty}$, defined in relation $(\ref{11})$. Furthermore, the convergence order, $CO(\Delta x)$, in space of the proposed computational technique is estimated using the formula
       \begin{equation*}
        CO(\Delta x)=\frac{\log\left(\frac{\||e_{(3\Delta x)}|\|_{0,\infty}}{\||e_{(\Delta x)}|\|_{0,\infty}}\right)}{\log(3)},
       \end{equation*}
        where, $e_{(\Delta x)}$ and $e_{(3\Delta x)}$ are the spatial errors associated with the mesh sizes $\Delta x$ and $3\Delta x$, respectively, whereas the temporal convergence rate, $CO(k)$, is calculated as follows
       \begin{equation*}
        CO(k)=\frac{\log\left(\frac{\||e_{(3k)}|\|_{0,\infty}}{\||e_{(k)}|\|_{0,\infty}}\right)}{\log(3)},
       \end{equation*}
        where $e_{(3k)}$ and $e_{(k)}$ denote the errors in time corresponding to time steps $3k$ and $k$, respectively. Lastly, the numerical computations are performed using MATLAB R$2007b$.\\

            $\bullet$ \textbf{Example 1 (Radially-symmetrical paraboloid)}\cite{48tdc}. Suppose that $\Omega=[0,l]\times[0,l]$ is the fluid region and $T=\frac{6\pi}{\omega}$ is the final time, where $l=4m$, $\omega=\frac{\sqrt{8gh_{0}}}{d}$ is the frequency, $h_{0}=0.1m$ denotes the water depth at the central point of the domain for a zero elevation, $d=1m$ represents the distance from the central point to the zero elevation of the shoreline (see Figure $\ref{fig1}$, \textbf{Figure 1.iv}) and $g=10m/s^{2}$ is the gravitational acceleration. The solution is periodic with no friction and the topography is a paraboloid of revolution defined as
            \begin{equation}\label{64}
            z(x,y)=h_{0}(r^{2}d^{-2}-1),
            \end{equation}
            where $r^{2}=(x-\frac{l}{2})^{2}+(y-\frac{l}{2})^{2}$, for $(x,y)\in[0,l]\times[0,l]$. The exact solution is defined as
            \begin{eqnarray*}
              h(x,y,t) &=& h_{0}\left[\frac{\sqrt{1-R^{2}}}{1-R\cos(\omega t)}-\frac{r^{2}}{d^{2}}\left(\frac{1-R^{2}}{(1-R\cos(\omega t))^{2}}-1\right)-1\right]-z(x,y), \\
              u(x,y,t) &=& \frac{\omega R}{2(1-R\cos(\omega t))}(x-\frac{l}{2})\sin(\omega t),\\
              v(x,y,t) &=& \frac{\omega R}{2(1-R\cos(\omega t))}(y-\frac{l}{2})\sin(\omega t),
            \end{eqnarray*}
         for every $(x,y,t)\in[0,l]\times[0,l]\times[0,T]$, where $R=\frac{d^{2}-r_{0}^{2}}{d^{2}+r_{0}^{2}}$, and $r_{0}=0.8m$ is the distance from the central point to the point where the shoreline is initially located (see Figure $\ref{fig1}$, \textbf{Figure 1.iv}). The initial and boundary conditions are directly obtained from the analytical solution.\\
         \text{\,}\\
         \textbf{Table 1.} $\label{T1}$ Analysis of convergence rate $CO(\Delta x)$ for the proposed time-split explicit/implicit approach with varying space step $\Delta x=\Delta y$ and time step $k=\Delta t$, satisfying restriction $(\ref{43})$, with $\gamma=18$.
          \begin{equation*}
          \begin{array}{c c}
          \text{\,Developed approach,\,\,where\,\,}k=3^{-6}& \\
           \begin{tabular}{ccccccc}
            \hline
            $\Delta x$ &  $\||\overline{h}-h|\|_{0,\infty}$ & $CO(\Delta x)$ & $\||\overline{u}-u|\|_{0,\infty}$ & $CO(\Delta x)$ & $\||\overline{v}-v|\|_{0,\infty}$ & $CO(\Delta x)$\\
             \hline
            $3^{-2}$ &  $2.0413\times10^{-2}$ &  ---  &  $7.9837\times10^{-2}$  &  ---   &  $6.5029\times10^{-2}$ & --- \\

            $3^{-3}$ &  $2.8230\times10^{-4}$ & 3.8967 & $1.1250\times10^{-4}$  & 3.8796 &  $9.1576\times10^{-4}$ & 3.8802 \\

            $3^{-4}$ &  $4.3440\times10^{-6}$ & 3.7995 & $1.5706\times10^{-6}$  & 3.8809 &  $1.2923\times10^{-5}$ & 3.8783 \\

            $3^{-5}$ &  $5.4240\times10^{-8}$ & 3.9897 & $1.9819\times10^{-8}$  & 3.9801 &  $1.6352\times10^{-7}$ & 3.9776 \\

            $3^{-6}$ & $6.1477\times10^{-10}$ & 4.0778 & $2.4538\times10^{-10}$  & 3.9974 & $1.9974\times10^{-9}$ & 4.0097 \\
            \hline
          \end{tabular} &
          \end{array}
          \end{equation*}
          \text{\,}\\
          \textbf{Table 2.} $\label{T2}$ Analysis of convergence order $CO(k)$ of the new algorithm with varying mesh space $\Delta x=\Delta y$ and time step $k=\Delta t$,  satisfying condition $(\ref{43})$, with $\gamma=1$.
          \begin{equation*}
          \begin{array}{c c}
          \text{\,Proposed computational scheme,\,\,where\,\,}\Delta x=\Delta y=3^{-5}& \\
           \begin{tabular}{ccccccc}
            \hline
            $k$ &  $\||\overline{h}-h|\|_{0,\infty}$ & $CO(k)$ & $\||\overline{u}-u|\|_{0,\infty}$ & $CO(k)$ & $\||\overline{v}-v|\|_{0,\infty}$ & $CO(k)$\\
             \hline
            $3^{-4}$ &  $2.7513\times10^{-3}$ &  ---  &  $1.9271\times10^{-2}$  &  ---    &  $3.0011\times10^{-2}$ & ---\\

            $3^{-5}$ &  $3.4508\times10^{-4}$ & 1.8897 & $2.3891\times10^{-3}$  & 1.9003  &  $3.8300\times10^{-3}$ &  1.8739\\

            $3^{-6}$ &  $3.9082\times10^{-5}$ & 1.9826 & $2.5832\times10^{-4}$  & 1.9657  &  $4.8091\times10^{-4}$ &  1.8887\\

            $3^{-7}$ &  $4.2067\times10^{-6}$ & 2.0289 & $2.6169\times10^{-5}$  & 2.0841  &  $5.7642\times10^{-5}$ & 1.9310 \\

            $3^{-8}$ &  $4.3649\times10^{-7}$ & 2.0623 & $2.9378\times10^{-6}$  & 1.9906  &  $6.5671\times10^{-6}$ &  1.9772\\
            \hline
          \end{tabular} &
          \end{array}
          \end{equation*}
          \text{\,}\\
          Tables 1 $\&$ 2 suggest that the proposed approach is second-order accurate in time and spatial fourth-order convergent.\\
          \text{\,}\\
          \text{\,}\\
           $\bullet$ \textbf{Example 2 (Planar surface in a paraboloid)}\cite{48tdc}. Let $\Omega=[0,l]\times[0,l]$ be the fluid region and $[0,T]=[0,\frac{6\pi}{\omega}]$ be the time interval, where $l=4m$, $\omega=\frac{\sqrt{2gh_{0}}}{d}$ is the frequency, $h_{0}=0.1m$, $d=1m$ and $g=10m/s^{2}$ is the acceleration of gravity. In this example, the moving shoreline represents a circle (see Figure $\ref{fig1}$, \textbf{Figure 1.iv}) and the topography is defined by equation $(\ref{64})$. The analytical solution is given by
            \begin{eqnarray*}
              h(x,y,t) &=& \frac{\eta h_{0}}{d^{2}}\left[2(x-\frac{l}{2})\cos(\omega t)+2(y-\frac{l}{2})\sin(\omega t)-\eta\right]-z(x,y), \\
              u(x,y,t) &=& -\eta\omega\sin(\omega t),\\
              v(x,y,t) &=& \eta\omega\cos(\omega t),
            \end{eqnarray*}
         for every $(x,y,t)\in[0,l]\times[0,l]\times[0,T]$, where $\eta=0.5$. Furthermore, both initial and boundary conditions are determined from the exact solution.\\

         \textbf{Table 3.} $\label{T3}$ Analysis of convergence rate $CO(\Delta x)$ for the new computational approach with varying space step $\Delta x=\Delta y$ and time step $k=\Delta t$, satisfying requirement $(\ref{43})$, with $\gamma=12$.
          \begin{equation*}
          \begin{array}{c c}
          \text{\,Developed approach,\,\,where\,\,}k=3^{-5}& \\
           \begin{tabular}{ccccccc}
            \hline
            $\Delta x$ &  $\||\overline{h}-h|\|_{0,\infty}$ & $CO(\Delta x)$ & $\||\overline{u}-u|\|_{0,\infty}$ & $CO(\Delta x)$ & $\||\overline{v}-v|\|_{0,\infty}$ & $CO(\Delta x)$\\
             \hline
            $3^{-2}$ &  $8.1866\times10^{-2}$ &   --- &  $2.0731\times10^{-1}$  & ---      &  $2.8608\times10^{-1}$ &  \\

            $3^{-3}$ &  $1.1548\times10^{-3}$ & 3.8787 & $3.2598\times10^{-3}$  & 3.7798   &  $4.5029\times10^{-3}$ & 3.7789 \\

            $3^{-4}$ &  $1.6095\times10^{-5}$ & 3.8896 & $4.8026\times10^{-5}$  & 3.8391   &  $6.6333\times10^{-5}$ & 3.8392\\

            $3^{-5}$ &  $2.1407\times10^{-7}$ & 3.9322 & $6.7647\times10^{-7}$  & 3.8800   &  $9.3299\times10^{-7}$ & 3.8813\\

            $3^{-6}$ &  $2.3474\times10^{-9}$ & 4.1079 & $8.6836\times10^{-9}$  & 3.9645   &  $1.1979\times10^{-8}$ & 3.9643\\
            \hline
          \end{tabular} &
          \end{array}
          \end{equation*}
          \text{\,}\\
          \textbf{Table 4.} $\label{T4}$ Analysis of convergence order $CO(k)$ of the developed numerical technique with varying mesh space $\Delta x=\Delta y$ and time step $k=\Delta t$, satisfying estimate $(\ref{43})$, with $\gamma=3$.
          \begin{equation*}
          \begin{array}{c c}
          \text{\,Proposed computational scheme,\,\,where\,\,}\Delta x=\Delta y=3^{-5}& \\
           \begin{tabular}{ccccccc}
            \hline
            $k$ &  $\||\overline{h}-h|\|_{0,\infty}$ & $CO(k)$ & $\||\overline{u}-u|\|_{0,\infty}$ & $CO(k)$ & $\||\overline{v}-v|\|_{0,\infty}$ & $CO(k)$\\
             \hline
            $3^{-4}$ &  $3.1030\times10^{-2}$ &  ---  &  $2.8755\times10^{-2}$  & ---     &  $2.9827\times10^{-2}$ & ---\\

            $3^{-5}$ &  $3.9636\times10^{-3}$ & 1.8731 & $3.5750\times10^{-3}$  & 1.8977  &  $3.7104\times10^{-3}$ &  1.8972\\

            $3^{-6}$ &  $4.6855\times10^{-4}$ & 1.9436 & $4.3548\times10^{-4}$  & 1.9163  &  $4.3041\times10^{-4}$ &  1.9608\\

            $3^{-7}$ &  $5.0941\times10^{-5}$ & 2.0198 & $4.8381\times10^{-5}$  & 2.0001  &  $4.8171\times10^{-5}$ &  1.9934 \\

            $3^{-8}$ &  $4.8719\times10^{-6}$ & 2.1365 & $5.4512\times10^{-6}$  & 1.9873  &  $4.7949\times10^{-6}$ &  2.1001\\
            \hline
          \end{tabular} &
          \end{array}
          \end{equation*}
          \text{\,}\\
          The numerical solutions provided by the new technique $(\ref{31})$-$(\ref{40})$ obtained from $1$ to $50$ iterations, respectively, are displayed in Figures $\ref{fig3}$-$\ref{fig4}$. Various time steps $k=3^{-4},\cdots,3^{-8}$, obtained from the stability requirement $(\ref{43})$ as the steady flow cases and space steps $\Delta x=\Delta y= 3^{-2},\cdots,3^{-6},$ in the mesh grids are used. Figures $\ref{fig3}$-$\ref{fig4}$ indicate that both water depth and velocities wave propagate with almost a perfectly value at different positions while the associated errors tend to zero. Thus, the approximate solutions cannot grow with time whenever the time step should satisfy restriction $(\ref{43})$. Additionally, \textbf{Tables 1-4} suggest that the errors associated with both water depth and velocities are second-order in time and spatial fourth-order. This shows that the developed time-split linearized explicit/implicit approach $(\ref{31})$-$(\ref{40})$ is temporal second-order accurate and fourth-order convergent in space. Finally, \textbf{Tables 1-4} and Figures $\ref{fig3}$-$\ref{fig4}$ suggest that the numerical solutions do not increase with time and converge to the analytical one. More specifically, they indicate that stability for the constructed approach $(\ref{31})$-$(\ref{40})$ is subtle. It is not unconditionally unstable, but stability depends on the parameter $\Delta x$ along with the time step $k$.

        \subsection*{Floods analysis in the Logone-et-Chari subdivision}
        The mathematical model for this overland flow is as follows: the study case is a uniform catchment so called, the logone basin having a length of $1000km$, an elevation of $364m$, an average annual discharge equals $492m^{3}/s$ and whose the area is approximately $78000km^{2}$ which can be approximated with horizontal dimensions $80km\times1000km$. In addition, the distance to the next downstream station is $445km$ whereas the distance to the basin outlet equals $579km$ and the length of upstream main stem is $511km$. The surface roughness and shear stress are assumed invariant in space and time. The water depth at $t=0$ is specified as the initial condition in the downstream. The initial water depth on the upstream side on the dam for both wet and dry beds equals $10^{-1}m$, whereas the initial flow depth in the downstream with respect to wet and dry beds are assumed to be $1.76\times10^{-1}m$ and $1.4\times10^{-3}m$, respectively. The velocities $u$ and $v$ at the inflow boundary are considered to be zero at any time while the initial velocities are obtained from the following three cases: minimum annual discharge ($q_{x}=q_{y}=16m^{3}/s$), average annual discharge ($q_{x}=q_{y}=492m^{3}/s$) and maximum annual discharge ($q_{x}=q_{y}=2420m^{3}/s$) observed in 2013, using equations $u(x,y,0)=\frac{q_{x}(x,y,0)}{h_{0}}$ and $v(x,y,0)=\frac{q_{y}(x,y,0)}{h_{0}}$, where $q_{x}$ and $q_{y}$ represent the discharges in the $x$-direction and $y$-direction, respectively, while $h_{0}=h(x,y,0)\in\{1.76\times10^{-1},1.4\times10^{-3}\}$, denotes the downstream initial water depth. The bed slops are determined utilizing equations $S_{0_{x}}=\frac{\partial z}{\partial x}=2h_{0}(x-40)$ and $S_{0_{y}}=\frac{\partial z}{\partial y}=2h_{0}(y-500)$, in the $x$-direction and $y$-direction, respectively, where $z$ is defined by equation $(\ref{64})$. We use the following values in the simulations: $c_{0}=40m^{1/2}/s$ (dimensional constant), $\overline{n}=0.025s/m^{1/3}$ (manning's number) and $g=10m/s^{2}$ (acceleration of gravity).\\

        The mesh sizes in $x$-direction and $y$-direction are $\Delta x=8.89$ and $\Delta y=12.36$, respectively, while the time step $k=0.33$. The period of floods is represented by the time interval $[0,\text{\,}T]=[0,\text{\,}3]$ (time in month), which corresponds from the second half of July $2024$ to the second half of October $2024$. $l_{1}=80km$ and $l_{2}=1000km$ are the rod interval lengths in the $x$- and $y$- directions, respectively. The water depth and velocities provided by the new computational approach $(\ref{31})$-$(\ref{40})$ during the period of floods are displayed in Figure $\ref{fig5}$. For initial water depth $h_{0}=1.76\times10^{-1}m$ (with respect to dry bed) and initial velocities $u_{0}=v_{0}=90.91m/s$ (obtained from minimum discharge), the first figure (in Figure $\ref{fig5}$) shows that the water depth and velocities waves propagate with perfect values and attain their maximum: $h_{\max}=2\times10^{2}m$, $u_{\max}=1.696\times10^{6}m/s$ and $v_{\max}=1.727\times10^{6}m/s$ on $21$ September $2024$ which corresponds to a duration of two months and six days. Moreover, the peak of inundations is observed on September $21$, $2024$. In addition, for $h_{0}=1.4\times10^{-3}m$ (with respect to wet bed) and $u_{0}=v_{0}=11430m/s$ (obtained from minimum discharge), the fourth figure (see Figure $\ref{fig5}$) indicates that water depth and velocities waves move with almost perfect constant values until September $21$ (corresponding to two months and six days) whereas both water depth and $y$-direction velocity exponentially increase and tend to $3.3\times10^{129}m$ and $0.5\times10^{129}m/s$, respectively, on the time interval $[2.3,\text{\,}2.9)$. Additionally, for large values of initial velocities, the other figures provided in Figure $\ref{fig5}$ suggest that the approximate solutions start to destroy after a fixed date different from September $21$, $2024$. These observations indicate that small values of initial water depth or large values of initial velocities must cause small numerical oscillations which can destroy the computed solutions. Specifically, suitable initial conditions and time steps that satisfy the stability limitation $(\ref{43})$, should generate efficient approximate solutions which help to assess and predict both water depth and velocities during different seasons and thus, investigating the minimum and maximum flood extents which represent the main tools to forecast several potential inundated locations in far north region of cameroon. However, condition $(\ref{43})$ deals with the computed solutions. Thus, physical insight must be used when the stability limitation $(\ref{43})$ of the new computational technique is investigated.

       \section{General conclusions and future works}\label{sec5}
        This paper has proposed a time-split linearized explicit/implicit approach for solving a two-dimensional shallow water model. A suitable time step restriction for stability of the developed computational technique is deeply analyzed using the $L^{\infty}(0,T;\text{\,}L^{2})$-norm while the convergence order of the new algorithm is numerically calculated. The graphs (Figures $\ref{fig3}$-$\ref{fig4}$) indicate that the numerical scheme $(\ref{31})$-$(\ref{40})$ is stable while \textbf{Tables 1-4} suggest that constructed time-split linearized explicit/implicit approach is second-order accurate in time and fourth-order convergent in space. Both tables and figures show that the computed solutions do not increase with time and converge to the analytical one. Furthermore, for appropriate initial conditions and time step satisfying restriction $(\ref{43})$, Figure $\ref{fig5}$ indicates that the water depth and velocities waves propagate with perfect values and the peak of floods is observed on $21$ September $2024$ whereas for small initial water depth or time step which does not satisfy limitation $(\ref{43})$, small numerical oscillations should destroy the computed solutions. Thus, time steps satisfying the stability restriction $(\ref{43})$ together with suitable initial water depth must generate good water depth and velocities during different seasons which will allow to predict the minimum and maximum flood extents in Cameroon far north region. Our future works will develop a time-split Lax-Wendroff/Crank-Nicolson technique in an approximate solution of a three-dimensional tectonic deformation problem.

      \subsection*{Ethical Approval}
     Not applicable.
     \subsection*{Availability of supporting data}
     Not applicable.
     \subsection*{Declaration of Interest Statement}
     The author declares that he has no conflict of interests.
     \subsection*{Funding}
     Not applicable.
     \subsection*{Authors' contributions}
     The whole work has been carried out by the author.

          \begin{figure}
         \begin{center}
         Stability of time-split linearized explicit/implicit for 2d-hydrodynamic flow.
         \begin{tabular}{c c}
         \psfig{file=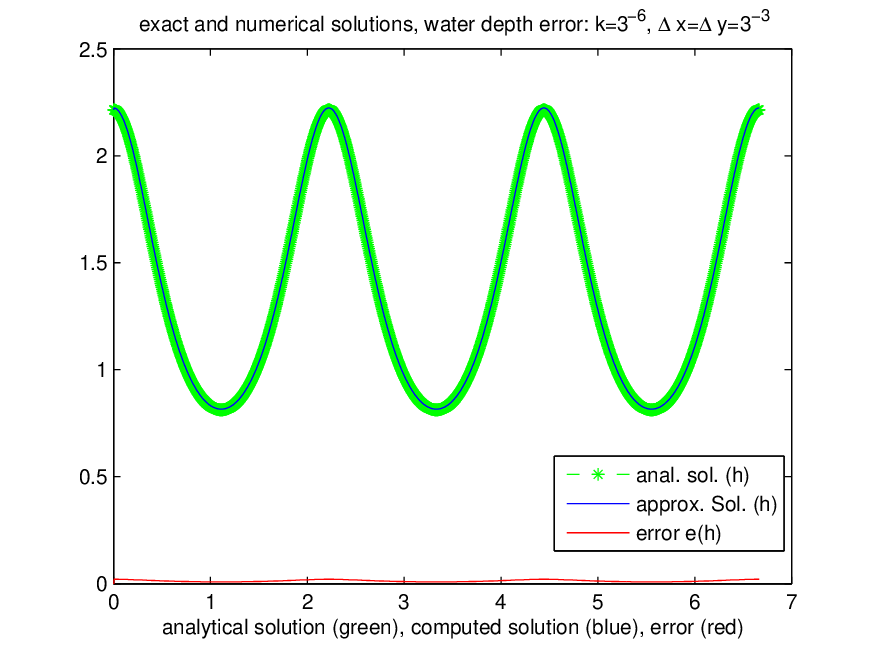,width=7cm} & \psfig{file=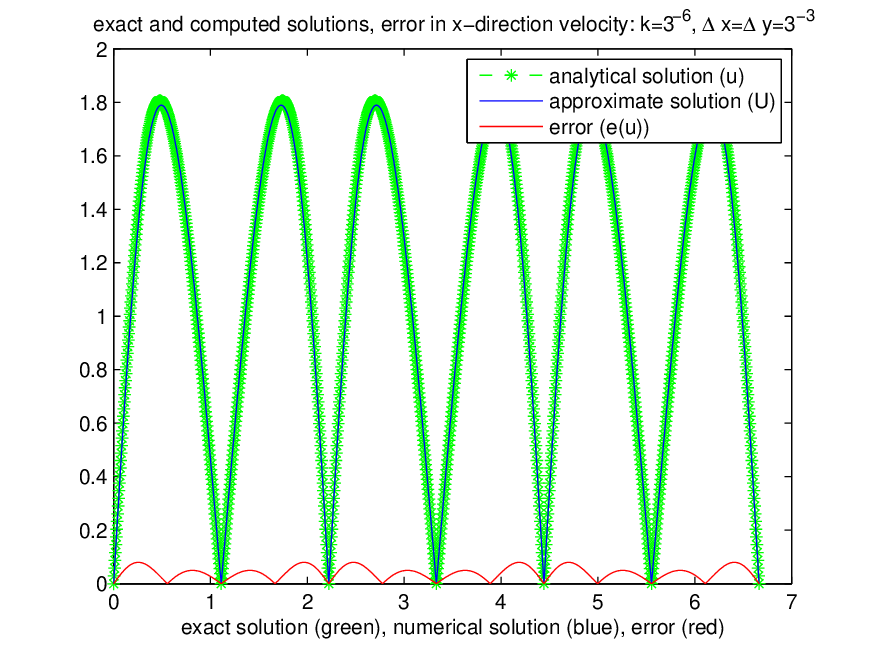,width=7cm}\\
         \psfig{file=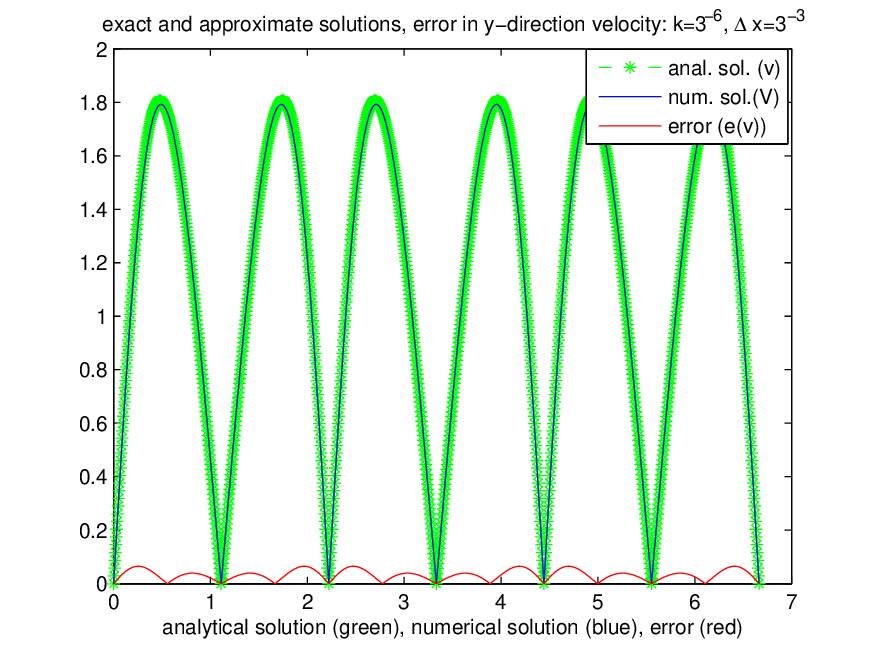,width=7cm} & \psfig{file=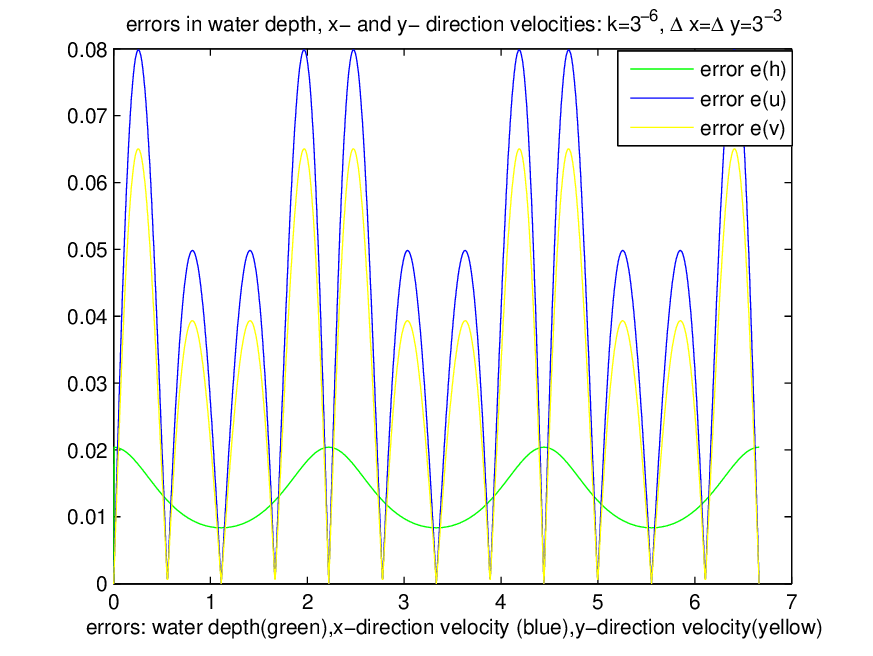,width=7cm}\\
         \psfig{file=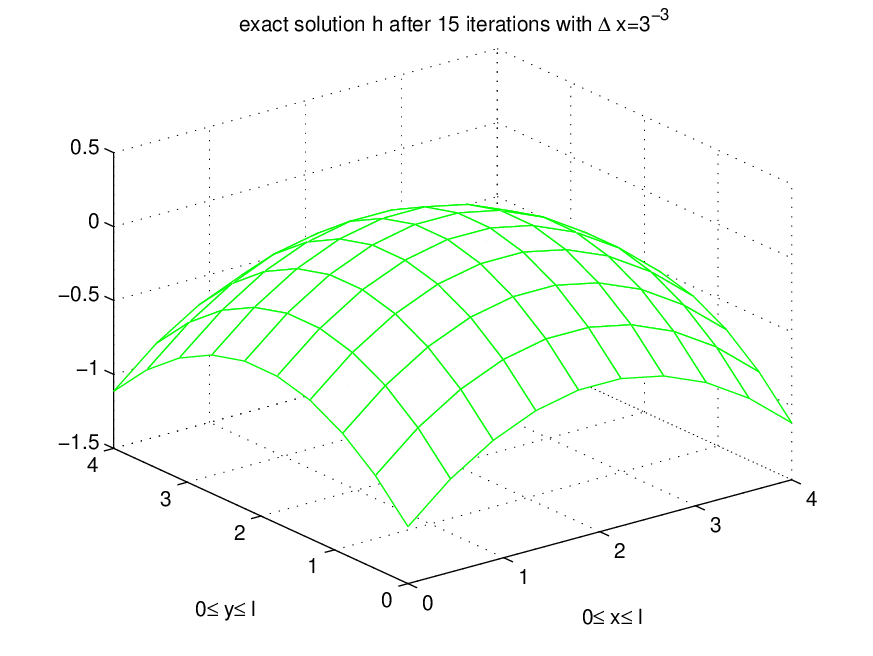,width=7cm} & \psfig{file=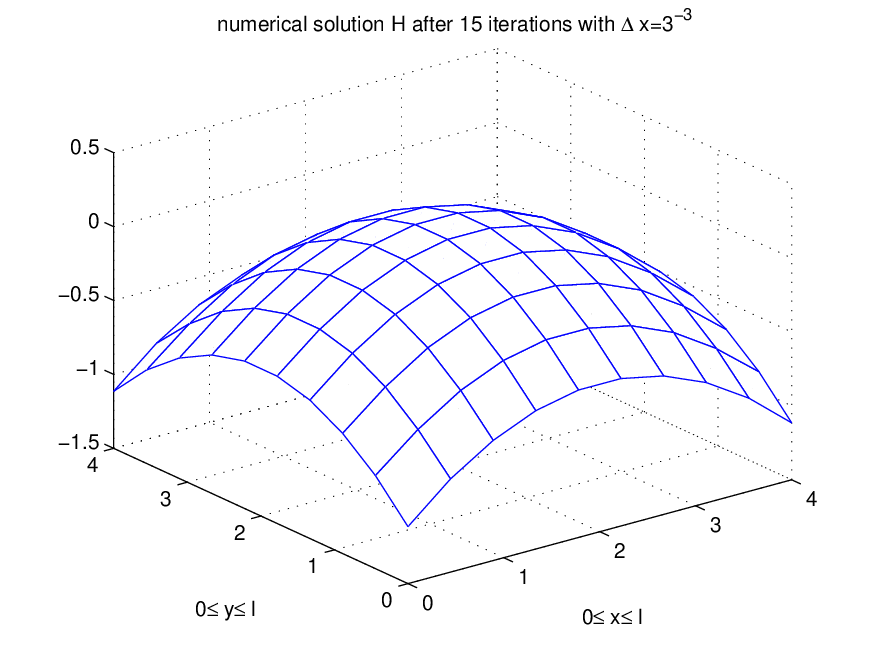,width=7cm}\\
         \psfig{file=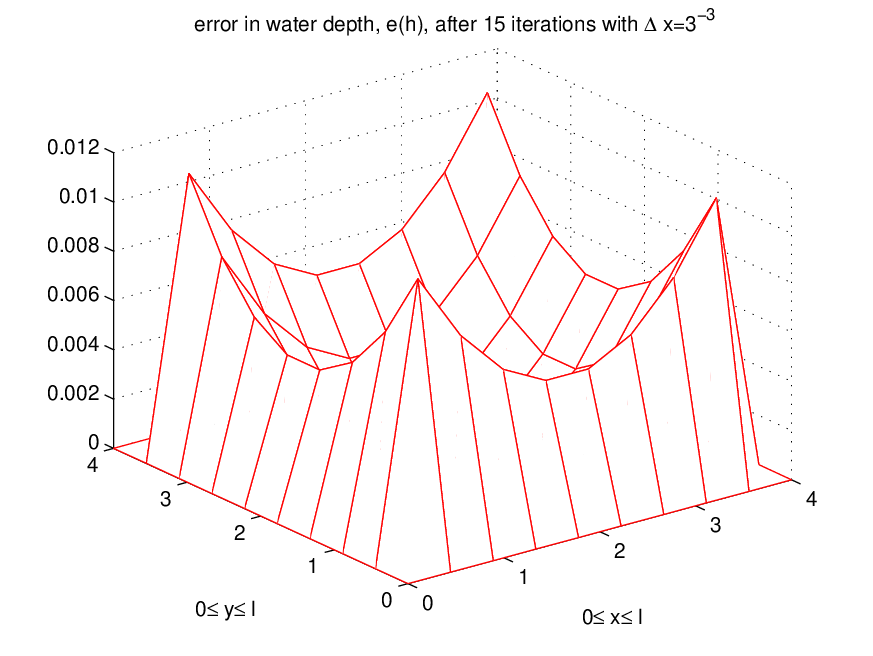,width=7cm} & \psfig{file=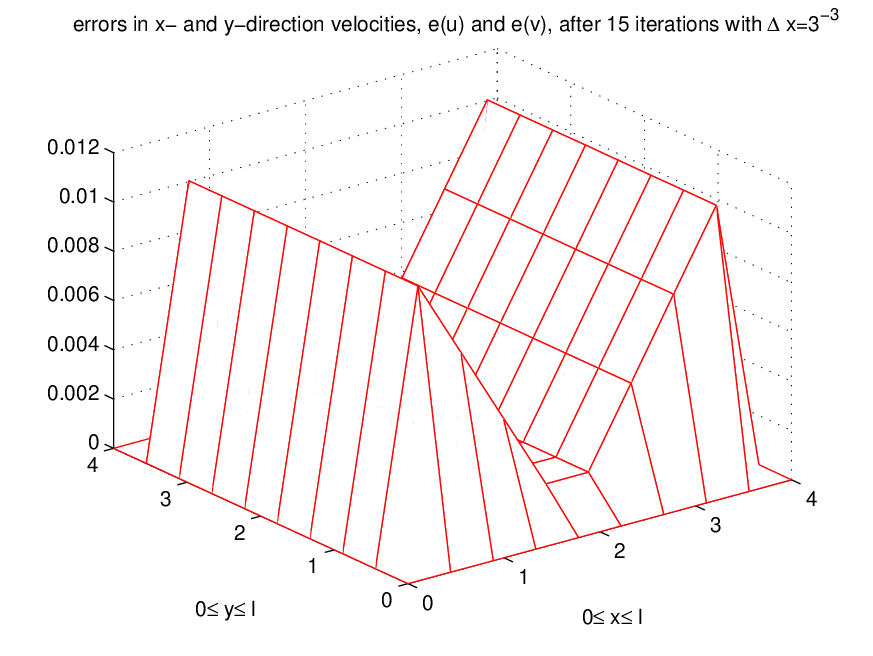,width=7cm}\\
           $ 0\leq x(m)\leq 4$ & $0\leq y(m)\leq 4$
         \end{tabular}
        \end{center}
        \caption{Graphs of water depth, x- and y-direction velocities and errors corresponding to Example 1.}
        \label{fig3}
        \end{figure}

    \begin{figure}
     \begin{center}
       Stability of time-split linearized explicit/implicit for 2d-hydrodynamic flow.
      \begin{tabular}{c c}
         \psfig{file=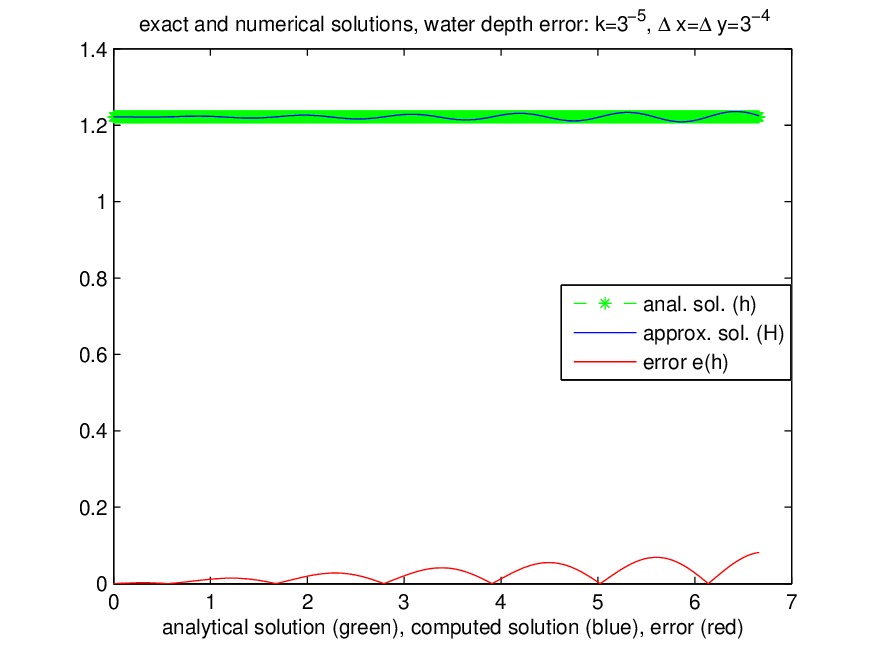,width=7cm} & \psfig{file=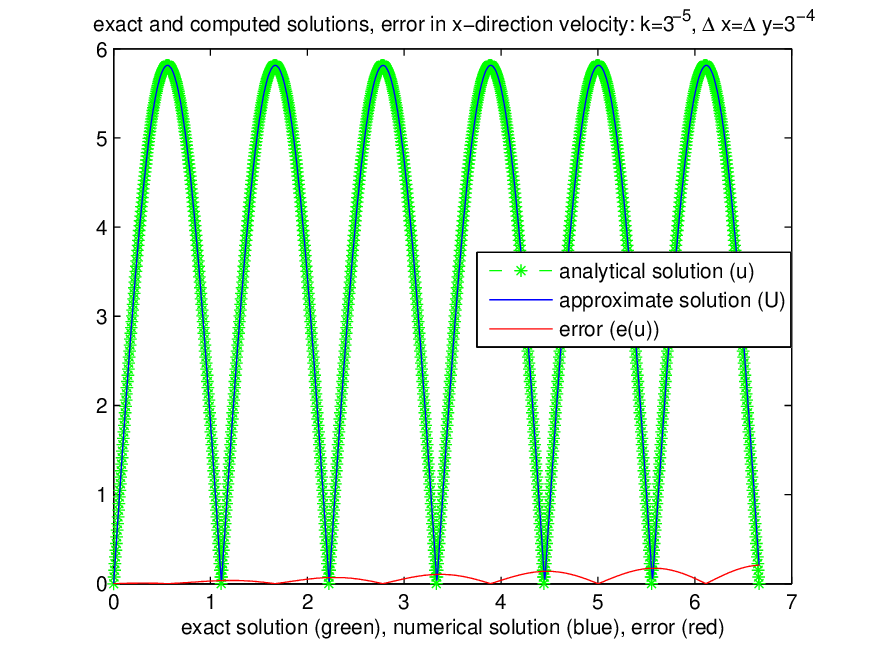,width=7cm}\\
         \psfig{file=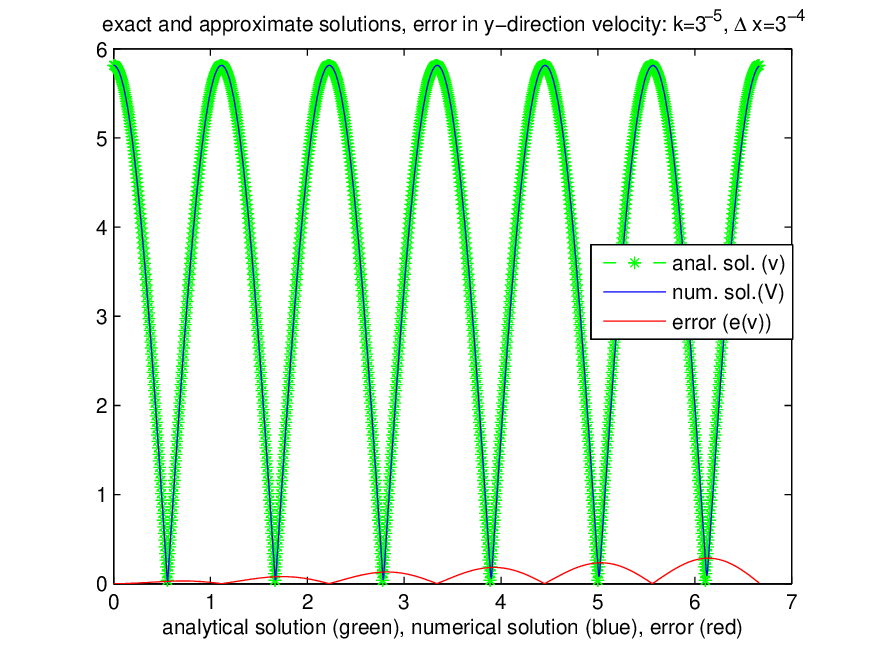,width=7cm} & \psfig{file=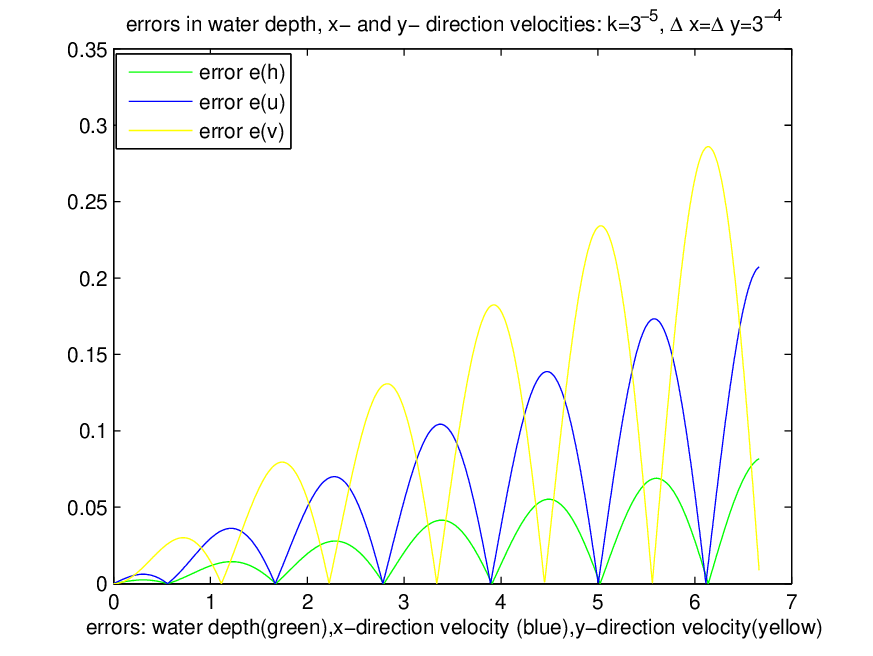,width=7cm}\\
         \psfig{file=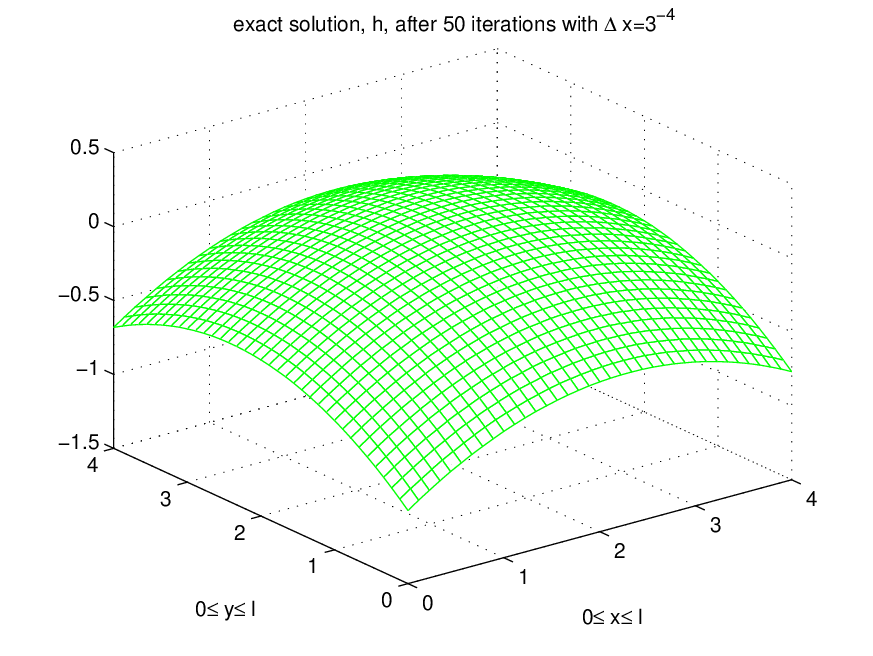,width=7cm} & \psfig{file=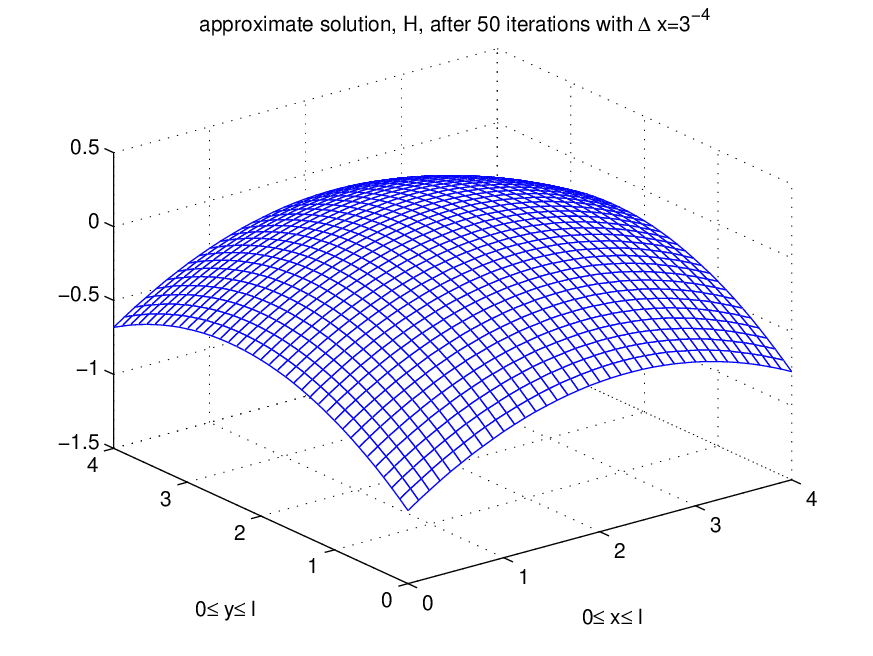,width=7cm}\\
         \psfig{file=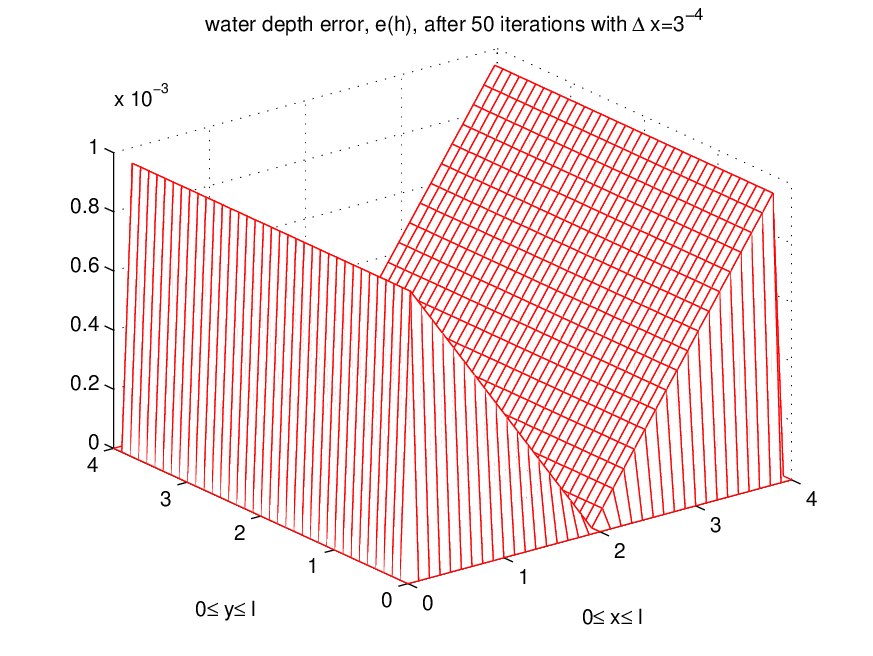,width=7cm} & \psfig{file=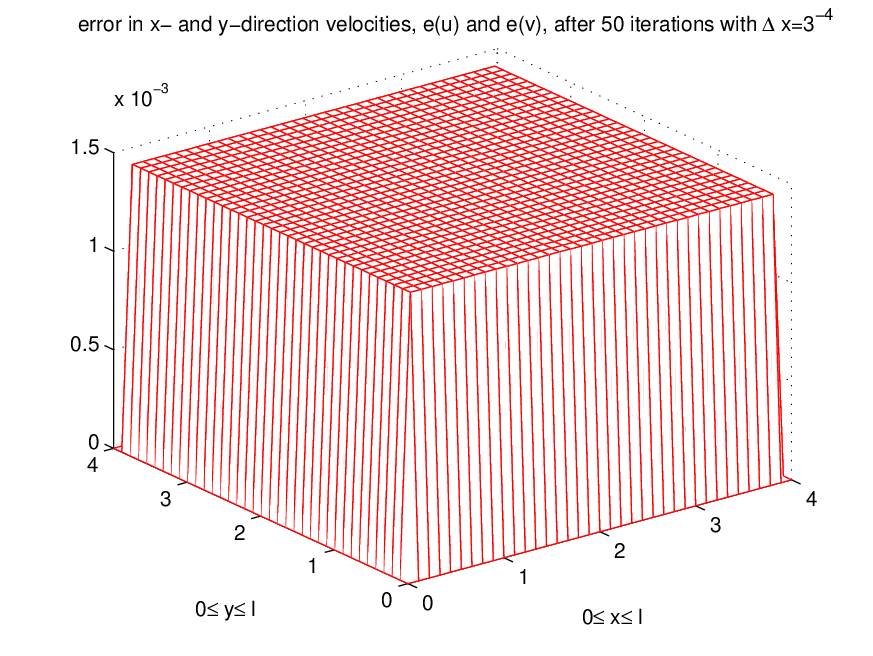,width=7cm}\\
           $ 0\leq x(m)\leq 4$ & $0\leq y(m)\leq 4$
         \end{tabular}
        \end{center}
         \caption{Graphs of water depth, x- and y-direction velocities and errors corresponding to Example 2.}
          \label{fig4}
          \end{figure}

      \begin{figure}
     \begin{center}
       Analysis of floods in the Logone river with various initial conditions.
      \begin{tabular}{c c}
         \psfig{file=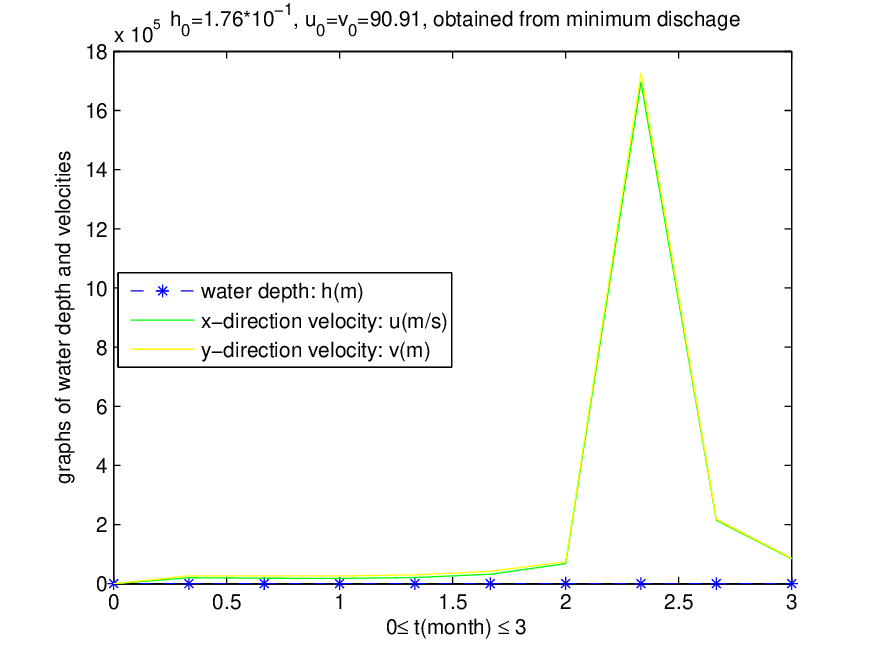,width=5.5cm} & \psfig{file=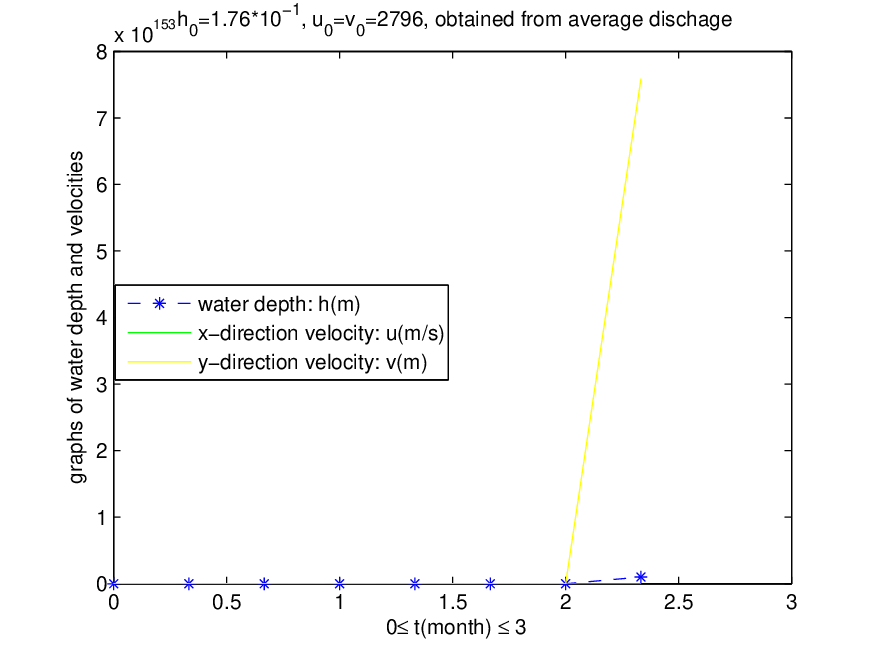,width=5.5cm}\\
         \psfig{file=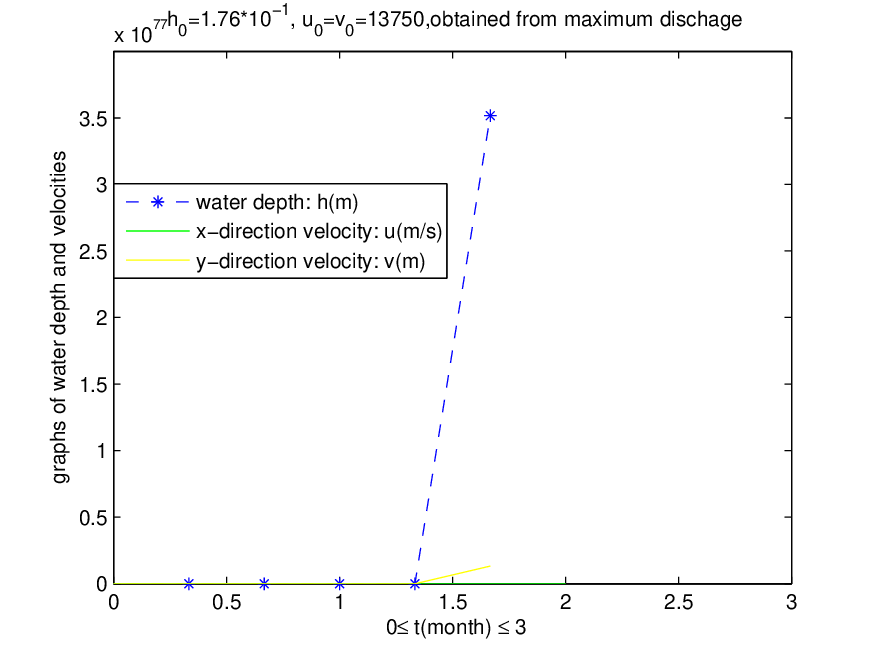,width=5.5cm} & \psfig{file=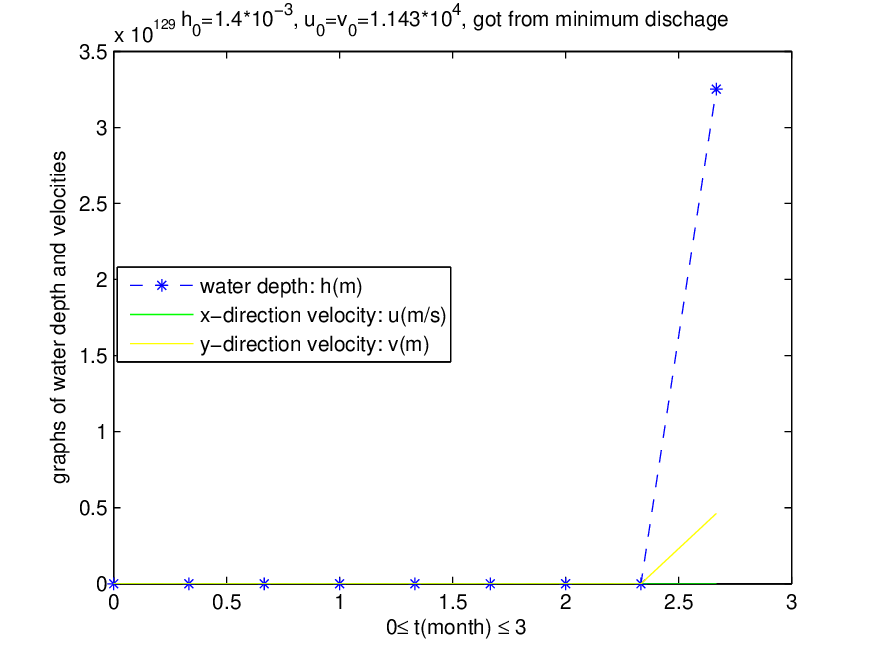,width=5.5cm}\\
         \psfig{file=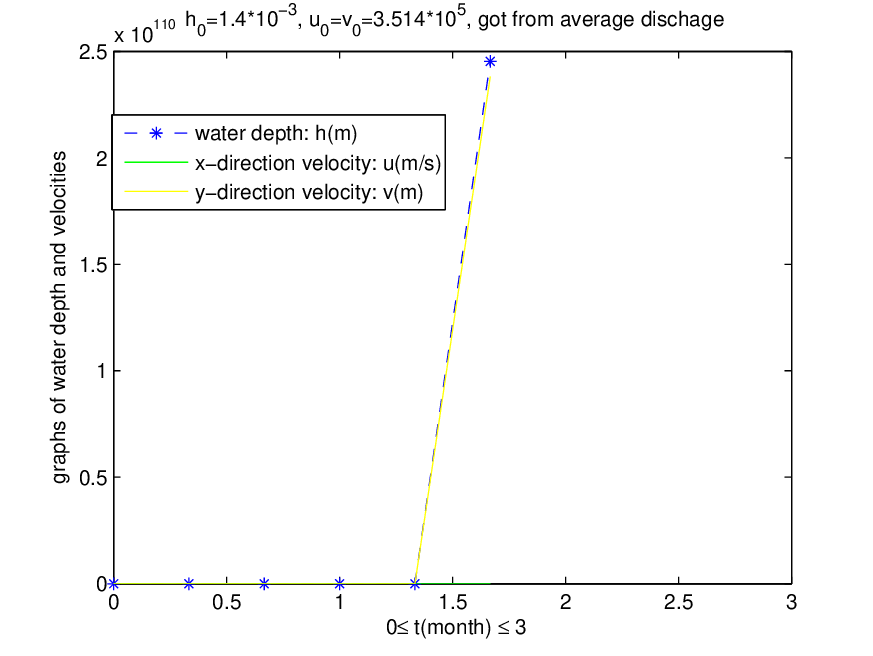,width=5.5cm} & \psfig{file=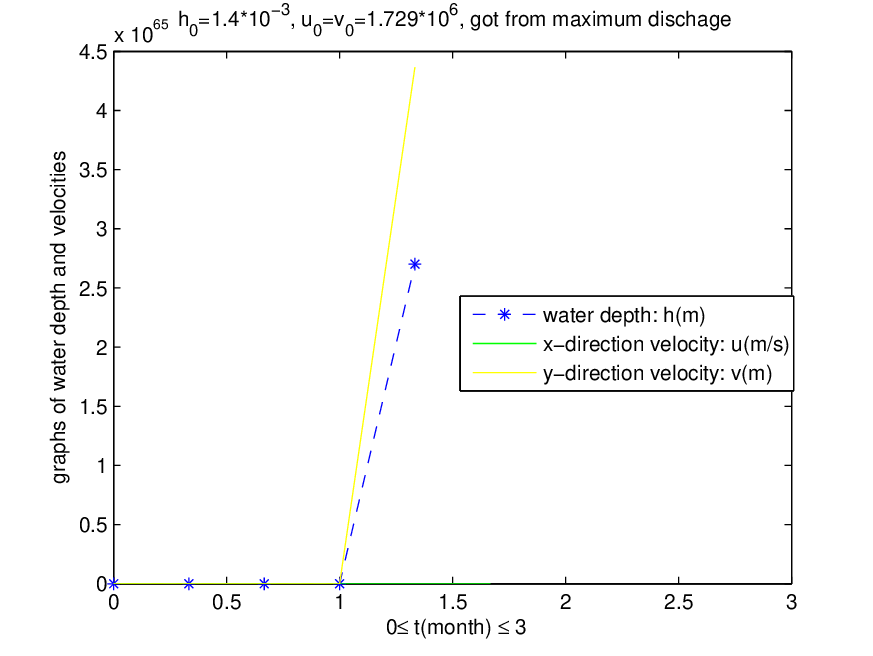,width=5.5cm}\\
         \psfig{file=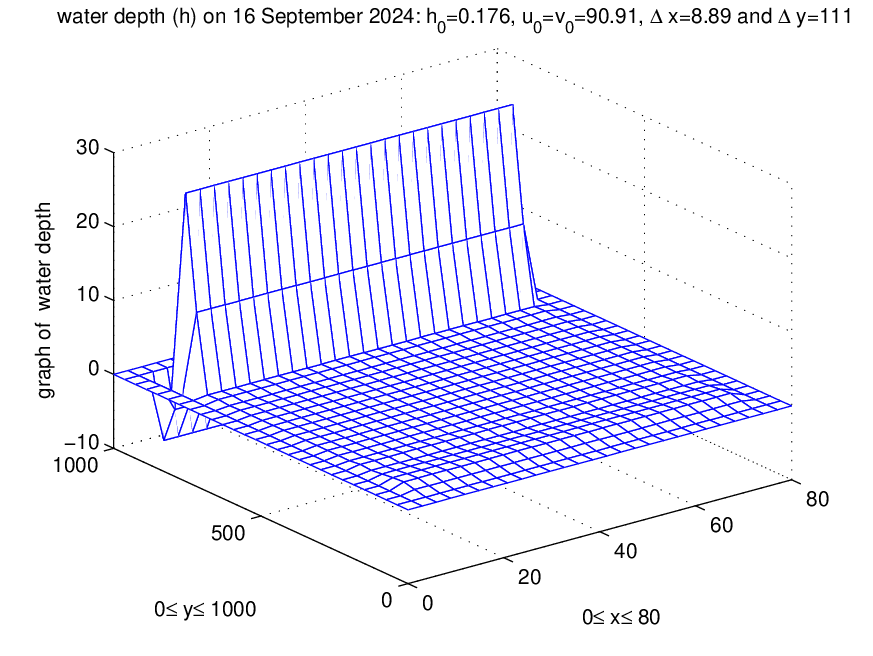,width=5.5cm} & \psfig{file=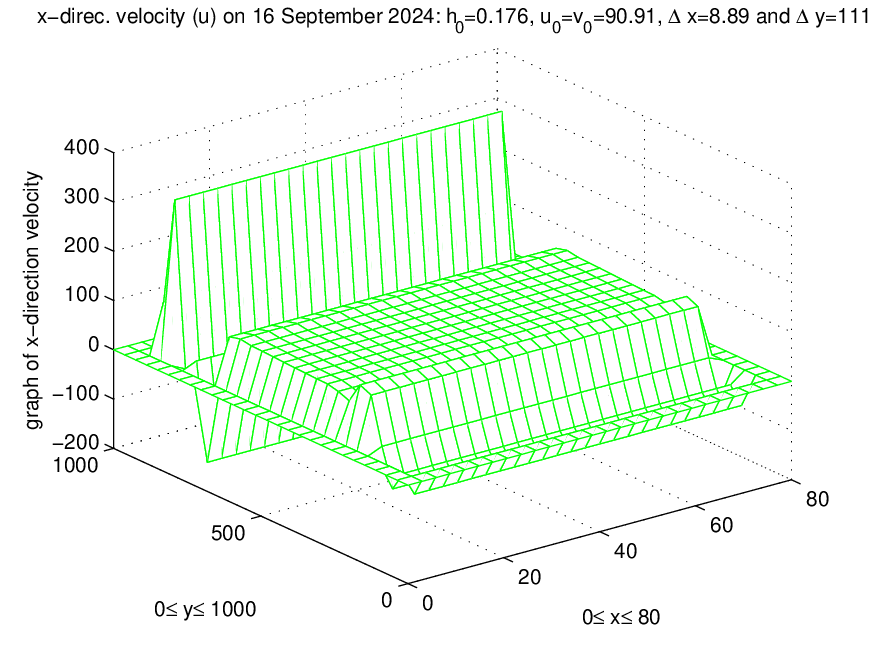,width=5.5cm}\\
         \psfig{file=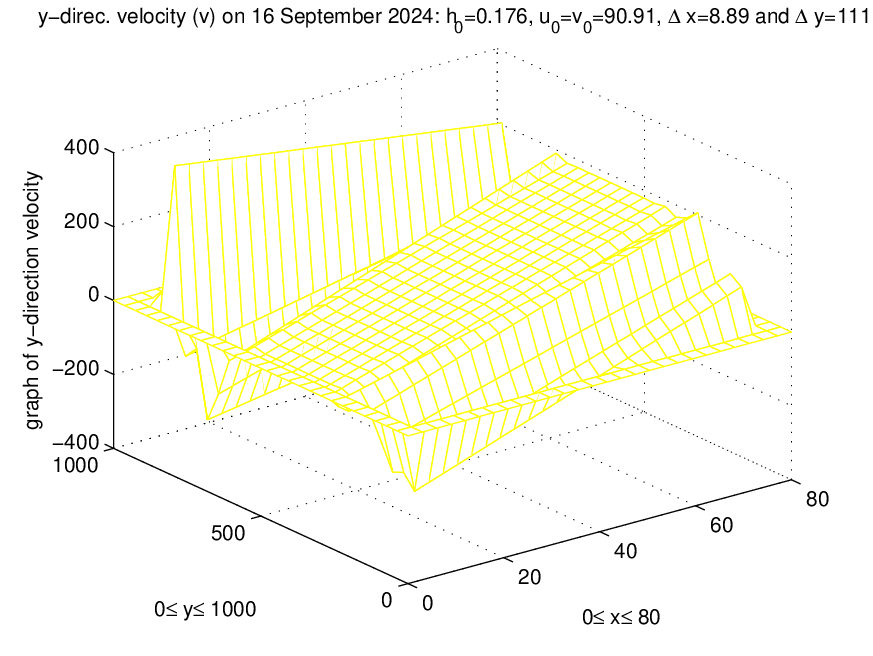,width=5.5cm} &
         \end{tabular}
        \end{center}
         \caption{Graphs of water depth, x- and y-direction velocities for floods in Logone river.}
          \label{fig5}
          \end{figure}

\begin{thebibliography}{99}

     \bibitem{15db}
     S. Abarbaned, A. Kumar. "Compact high-order schemes for the Euler equations", J. Scientific Computing, $3(1988)$, $275$-$288$.

     \bibitem{19yzw}
     M.J. Ablowitz, B. M. Herbst. "On the numerical solution of the sine-Gordon equation. 1. Integrable discretizations and homoclinic manifolds", J. Comput. Phys., $126(1996)$, $299$-$314$.

     \bibitem{8yzw}
     K. O. Aiyesimoju, R. J. Sobey. "Process splitting of the boundary conditions for the advection-dispersion equation", Int. J. Numer. Methods Fluids, $9(1989)$, $235$-$244$.

     \bibitem{1en}
      R. T. Alqahtani, J. C. Ntonga, E. Ngondiep. "Stability analysis and convergence rate of a two-step predictor-corrector approach for shallow water equations with source terms", AIMS Mathematics, $8(4)$ $(2023)$, $9265$-$9289$.

      \bibitem{33yzw}
      W. Bao, Q. Du. "Computing the ground state solution of Bose-Einstein condensates by a normalized gradient flow", SIAM J. Sci. Comput., $25(2004)$, $1674$-$1697$.

       \bibitem{91yzw}
       G. F. Carey, Y. Shen. "Least-squres finite element approximation of Fisher's reaction diffusion equation", Numer. Meth. partial Differential Equations, $11(1995)$, $175$-$186$.

      \bibitem{5yzw}
      L. Demkowicz, T. J. Oden, W. Rachowicz. "A new finite element method for solving compressible Navier-Stokes equations based on an operator splitting method and h-p adaptivity", Comput. Methods Appl. Mech. Engng., $84(1990)$, $275$-$326$.

      \bibitem{11yzw}
      S. Descombes, M. Massot. "Operator splitting for nonlinear reaction-diffusion systems with an entropic structure: singular perturbation and order reduction", Numer. Math., $97(2004)$, $667$-$698$.

     \bibitem{fr}
      F.R. Fiedler, J.A. Ramirez. "A numerical method for simulating discontinuous shallow flow over an infiltrating surface", Int. J. Numer. Meth. Fluids, $32:$ $219$-$240,$ $(2000)$.

      \bibitem{26db}
       R. Garcia, R. A. Kahawaita. "Numerical solution of the St. Venant equations with MacCormack finite-difference scheme", Int. J. Numer. Meth. Fluids, $6(5)$ $(1986)$, $259$-$274$.

       \bibitem{13yzw}
       D. Goldman, T. J. kaper. "Nth-order operator splitting schemes and nonreversible systems", SIAM J. Numer. Anal., $33(1996)$, $349$-$367$.

       \bibitem{42yzw}
       D. Gottlieb, J. S. Hesthaven. "Spectral methods for hyperbolic problems", J. Comput. Appl. Math., $128(2001)$, $83$-$131$.

       \bibitem{41yzw}
       D. Gottlieb, E, tadmor. "The CFL condition for spectral approximations to hyperbolic initial-boundary value problems", Math. Comput., $56(1991)$, $565$-$588$.

      \bibitem{16db}
      M. M. Gupta. "High accuracy solutions of incompressible Navier-Stokes equations", J. Comput. Phys., $93(1991)$, $343$-$357$.

     \bibitem{hernonin2013}
      J. Henonin, B. Russo, O. Mark, P. Gourbesville. "Real-time urban flood forecasting and modelling-a state of the art", J. Hydroinformatics $15(3)$ $(2013)$, $717$-$736$.

     \bibitem{23db}
      R. Hixon, E. Turkel. "Compact implicit MacCormack type schemes with high accuracy", J. Comput. Phys., $158$ $51$-$70$ $(2000).$

      \bibitem{27yzw}
      H. Holden, K. H. Karlsen, N. H. Risebro. "Operator splitting methods for generalized Korteweg-de Vries equations", J. Comput. Phys., $153(1999)$, $203$-$222$.

      \bibitem{9yzw}
      L. A. Khan, P. L. -F. Liu. "Numerical analyses of operator-splitting algorithms for the two-dimensional advection-diffusion equation", Comput. Meth. Appl. Mech. Engng., $152(1998)$, $337$-$359$.

     \bibitem{hok}
      H. O. Kreiss. "On difference approximations of the dissipative type for hyperbolic differential equations", Comm. Pure Appl. Math., $17(1964)$ $335$-$353,$
      MR $29\#4210$.

      \bibitem{29yzw}
      J. Lee, B. Fornberg. "A split step approach for the 3-D Maxwell's equations", J. Comput. Appl. Math., $158(2003)$, $485$-$505$.

      \bibitem{17db}
       S. K. Lele. "Compact finite difference schemes with spectral-like resolution", J. Comput. Phys., $103(1992)$, $16$-$42$.

      \bibitem{7yzw}
       R. J. LeVeque. "Intermediate boundary conditions for time split methods applied to hyperbolic partial differential equations", Math. Comput., $47(1986)$, $37$-$54$.

       \bibitem{18db}
       M. Li, T. Tang. "A compact fourth-order finite difference scheme for the steady incompressible Navier-Stokes equations", Int. J. Numer. Meth. Fluids, $20(1995)$, $1137$-$1151$.

       \bibitem{19db}
       M. Li, T. Tang. "A compact fourth-order finite difference scheme for unsteady viscous incompressible flows", J. Scientific Computing, $16(2001)$, $29$-$45$.

       \bibitem{21db}
      R. W. Maccormack. "The effect of viscosity in hypervelocity impact cratering", AIAA, $69$-$354,$ $(1969).$

      \bibitem{28yzw}
      G. M. Muslu, H. A. Erbay. "A split-step Fourier method for the complex modified Korteweg-de Vries equation", Comput. Math. Appl., $45(2003)$, $503$-$514$.

    \bibitem{3en}
     E. Ngondiep. "An efficient numerical approach for solving three-dimensional Black-Scholes equation with stochastic volatility", Math. Meth. Appl. Sci., $(2024)$, $1$-$21$, DOI $10.1002$/mma.$10576$.

    \bibitem{5en}
     E. Ngondiep. "An efficient high-order two-level explicit/implicit numerical scheme for two-dimensional time fractional mobile/immobile advection-dispersion model", Int. J. Numer. Meth. Fluids, $96(8)$ $(2024)$, $1305$-$1336$.

    \bibitem{10en}
    E. Ngondiep. "Stability analysis of MacCormack rapid solver method for evolutionary Stokes-Darcy problem", J. Comput. Appl. Math., $345(2019)$, $269$-$285$.

    \bibitem{11en}
     E. Ngondiep. "A six-level time-split Leap-Frog/Crank-Nicolson approach for two-dimensional nonlinear time-dependent convection-diffusion-reaction equation", Int. J. Comput. Meth., $20(08)$ $(2023)$, $2250064$, Doi: $10.1142$/S$0219876222500645$.

     \bibitem{12en}
     E. Ngondiep. "An efficient three-level explicit time-split scheme for solving two-dimensional unsteady nonlinear coupled Burgers equations", Int. J. Numer. Methods Fluids, $92(4)$ $(2020)$, $266$-$284$.

     \bibitem{2en}
     E. Ngondiep. "Unconditional stability over long time intervals of a two-level coupled MacCormack/Crank-Nicolson method for evolutionary mixed Stokes-Darcy model", J. Comput. Appl. Math., $409(2022)$, $ 114148$, Doi: $10.1016$/j.cam.$2022.114148$.

    \bibitem{13en}
     E. Ngondiep. "Long time unconditional stability of a two-level hybrid method for nonstationary incompressible Navier-Stokes equations", J. Comput. Appl. Math. $345(2019)$, $501$-$514$.

    \bibitem{7en}
     E. Ngondiep. "A high-order numerical scheme for multidimensional convection-diffusion-reaction equation with time-fractional derivative", Numer. Algorithms, $94(2023)$, $681$-$700$.

     \bibitem{16en}
     E. Ngondiep. "A high-order combined finite element/interpolation approach for solving nonlinear multidimensional generalized Benjamin-Bona-Mahony-Burgers' equations", Math. Comput. Simul., $215(2024)$, $560$-$755$.

     \bibitem{9en}
     E. Ngondiep, A. N. Njomou, G. I. Mondinde. "A predictor-corrector approach to investigate and predict the dynamic of cytokine levels and human immune cell activation to Staphylococcus Aureus", Int. J. of Biomath., DOI: $10.1142$/S$1793524524501122$, $(2024)$, $1$-$29$.

    \bibitem{14en}
    E. Ngondiep. "Long time stability and convergence rate of MacCormack rapid solver method for nonstationary Stokes-Darcy problem", Comput. Math. Appl., $75$, $(2018)$, $3663$-$3684$.

    \bibitem{15en}
    E. Ngondiep. "A novel three-level time-split MacCormack scheme for two-dimensional evolutionary linear convection-diffusion-reaction equation with source term", Int. J. Comput. Math., $98(1)$ $(2021)$, $47$-$74$.

     \bibitem{8en}
     E. Ngondiep. "A robust numerical two-level second-order explicit approach to predict the spread of covid-$2019$ pandemic with undetected infectious cases", J. Comput. Appl. Math., $403(2022)$, DOI:10.1016/j.cam.2021.113852.

    \bibitem{17en}
    E. Ngondiep. "A posteriori error estimates of MacCormack rapid solver method for nonstationary incompressible Navier-Stokes equations", J. Comput. Appl. Math., $438(2024)$, $115569$.

    \bibitem{6en}
     E. Ngondiep. "A fast three-step second-order explicit numerical approach to investigating and forecasting the dynamic of corruption and poverty in Cameroon", Heliyon, DOI: $10.1016/$j.heliyon.$2024$.e$38236$, $10(19)$ $(2024)$.

     \bibitem{4en}
     E. Ngondiep. "A two-level fourth-order approach for time-fractional convection-diffusion-reaction equation with variable coefficients", Commun. Nonlinear Sci. Numer. Simul., $111(2022)$, $106444$, Doi: $10.1016$/j.cnsns.$2022.106444$.

     \bibitem{20db}
     S. K. Pandit, J. C. Kalita, D. C. Dalal. "A transient higher order compact scheme for incompressible viscous flows on geometries beyond rectangular", J. Comput. Phys., $225(2007)$, $1100$-$1124$.

     \bibitem{10yzw}
     B. Sportisse. "An analysis of operator splitting techniques in the stiff case", J. Comput. Phys., $161(2000)$, $140$-$168$.

    \bibitem{sb}
    J. Stoer, R. Bulirsch. "Introduction to numerical analysis, text in applied mathematics", second edition, Spring-Verlag, $(1991).$

    \bibitem{12yzw}
    G. Strang. "On the construction and comparison of difference schemes", SIAM J. Numer. Anal., $5(1968)$, $506$-$517$.

    \bibitem{26yzw}
    T. R. Taha, M. J. Abowitz. "Analytical and numerical aspects of certain nonlinear evolution equations, II. Numerical, nonlinear Schr\"{o}dinger equation", J. Comput. Phys., $55(1984)$, $231$-$253$.

    \bibitem{48tdc}
    W. C. Thacker. "Some exact solutions to the nonlinear shallow water wave equations", J. Fluid Mech., $107(1981)$, $499$-$508$.

    \bibitem{3fr}
    C. B. Vreugdenhil. "Numerical method for shaloww water flow", Kluwer Academic Publishers, Dordrecht $(1994)$.

    \bibitem{yu2014}
    P. S. Yu, T. C. Yang, C. M. Kuo, S. T. Chen. "Development of an integrated computational tool to assess climate change impacts on water supply-demand and flood inundation", J. Hydroinformatics, $16(3)$ $(2014)$, $710$-$730$.

    \bibitem{2fr}
    W. Zhang, T. W. Cundy. "Modeling of two-dimensional overland flow", Water Resour. Res., $25(1989)$, $2019$-$2035$.

    \bibitem{8db}
    C. Zoppou, S. Roberts. "Numerical solution of the two-dimensional unsteady dam-break", Appl. Math. Mode., $24(2000)$, $457$-$475$.

     \bibitem{2lc}
     "Logone river/river, Africa", Encyclopedia Britannica. Retrieved, $2017$-$06$-$08$.

     \bibitem{3lc}
     "GRDC-Chari basin", Der Logone in Bongor, $2009$.

    \bibitem{4lc}
    "Cameroon: Floods-Oct 2013", Reliefweb, $04$ June $2014$. Retrieve, $2014$-$06$-$10$.

    \bibitem{unicef}
    "Unicef Cameroon flash update $N^{o}3$: Floods-Far North", $25$ September $2024$.

    \bibitem{2unicef}
     OCHA, flood snapshot as of $19$ September. "Cameroun: Extr\^{e}me-Nord-Aper\c{c}u des inondations (au $19$ Septembre $2024$)-Cameroon|ReliefWeb".

     \bibitem{4unicef}
     Regional Delegations of Basic and Secondary Education, $15$ September.

     \bibitem{5unicef}
     Regional Delegation of Public Health of the Far North, presented on $13$ September $2024$ in Maroua, Far North.

     \bibitem{6unicef}
     Centre for humdata, based on data from WorldPop, FloodScan, HDX, ECMWF, on $19$ September.


     \end{thebibliography}
     \end{document}